\documentclass[english,letter paper,12pt,reqno]{article}
\usepackage{etex} 
\usepackage{array}
\usepackage{varwidth}
\usepackage{bussproofs}
\usepackage{epigraph}
\usepackage{stmaryrd}
\usepackage{mathdots}
\usepackage{amsmath, amscd, amssymb, mathrsfs, accents, amsfonts,amsthm}
\usepackage[all]{xy}
\usepackage{mathtools} 
\usepackage{tikz-cd}
\usetikzlibrary{calc}
\usetikzlibrary{positioning}

\AtEndDocument{\bigskip{\footnotesize%
  James Clift \par
  \textsc{Department of Mathematics, University of Melbourne} \par
  \textit{E-mail address}: \texttt{j.clift3@student.unimelb.edu.au} \par
  \vspace{0.3cm}
  Daniel Murfet \par
  \textsc{Department of Mathematics, University of Melbourne} \par  
  \textit{E-mail address}: \texttt{d.murfet@unimelb.edu.au} \par
}}

\DeclarePairedDelimiter\ket{\lvert}{\rangle}
\DeclarePairedDelimiterX\braket[2]{\langle}{\rangle}{#1 \delimsize\vert #2}
\DeclarePairedDelimiterX\inner[2]{\langle}{\rangle}{#1,#2}

\definecolor{Myblue}{rgb}{0,0,0.6}
\usepackage[a4paper,colorlinks,citecolor=Myblue,linkcolor=Myblue,urlcolor=Myblue,pdfpagemode=None]{hyperref}

\SelectTips{cm}{}

\theoremstyle{definition}
\newtheorem{theorem}{Theorem}[section]
\newtheorem{proposition}[theorem]{Proposition}
\newtheorem{lemma}[theorem]{Lemma}
\newtheorem{corollary}[theorem]{Corollary}

	\newtheorem{definition}[theorem]{Definition}
	\newtheorem{example}[theorem]{Example}
	\newtheorem{remark}[theorem]{Remark}

\def\res{\operatorname{Res}}

\def\im{\operatorname{Im}}

\def\Hom{\operatorname{Hom}}

\def\vacu{\ket{\emptyset}}
\def\be{\begin{equation}}
\def\ee{\end{equation}}

\DeclareMathOperator{\Prim}{Prim}

\def\comp{\underline{\textup{comp}}}
\def\contract{\;\lrcorner\;}
\newcommand{\kk}{k} 
             \def\NN{\mathbb{N}}            
\newcommand{\Sum}{\sum\limits}
\newcommand{\Coalg}{\bold{Coalg}}

\makeatletter
\def\Ddots{\mathinner{\mkern1mu\raise\p@
\vbox{\kern7\p@\hbox{.}}\mkern2mu
\raise4\p@\hbox{.}\mkern2mu\raise7\p@\hbox{.}\mkern1mu}}
\makeatother

\begin{document}

\def\ScoreOverhang{1pt}

\makeatletter
\DeclareRobustCommand{\rvdots}{%
  \vbox{
    \baselineskip4\p@\lineskiplimit\z@
    \kern-\p@
    \hbox{}\hbox{.}\hbox{.}\hbox{.}
  }}
\makeatother

\newcommand{\proofvdots}[1]{\overset{\displaystyle #1}{\rvdots}}
\def\Res{\res\!}
\newcommand{\ud}[1]{\operatorname{d}\!{#1}}
\newcommand{\Ress}[1]{\res_{#1}\!}
\newcommand{\cat}[1]{\mathcal{#1}}
\newcommand{\lto}{\longrightarrow}
\newcommand{\xlto}[1]{\stackrel{#1}\lto}
\newcommand{\mf}[1]{\mathfrak{#1}}
\newcommand{\md}[1]{\mathscr{#1}}
\newcommand{\church}[1]{\underline{#1}}
\newcommand{\prf}[1]{\underline{#1}}
\newcommand{\den}[1]{\llbracket #1 \rrbracket}
\def\l{\,|\,}
\def\sgn{\textup{sgn}}
\def\cont{\operatorname{cont}}
\def\counit{\varepsilon}
\def\ptail{\underline{\operatorname{tail}}}
\def\phead{\underline{\operatorname{head}}}
\def\comp{\underline{\textup{comp}}}
\def\mult{\underline{\textup{mult}}}
\def\repeat{\underline{\textup{repeat}}}
\def\contract{\;\lrcorner\;}
\def\<{\langle} \def\>{\rangle}
\newcommand{\id}{\text{id}}
\newcommand{\del}{\partial}
\newcommand{\Inj}{\operatorname{Inj}}

\newcommand{\tTur}{\textbf{Tur}}
\newcommand{\tInt}{\textbf{int}}
\newcommand{\tBint}{\textbf{bint}}
\newcommand{\tList}{\textbf{list}}
\newcommand{\tTint}{\textbf{tint}}
\newcommand{\tBool}{\textbf{bool}}
\newcommand{\tMBool}{\textbf{${_m}$bool}}
\newcommand{\tNBool}{\textbf{${_n}$bool}}
\newcommand{\tSBool}{\textbf{${_s}$bool}}
\newcommand{\tHBool}{\textbf{${_h}$bool}}
\newcommand{\tSList}{{_s}\textbf{list}}

\newcommand{\pTapehead}{\text{\underline{tapehead}}}
\newcommand{\pAbsStep}{\text{\underline{absstep}}}
\newcommand{\pSymbol}{\text{\underline{symbol}}}
\newcommand{\pBoolWeak}{\text{\underline{boolweak}}}
\newcommand{\pLeft}{\text{\underline{left}}}
\newcommand{\pEval}{\text{\underline{eval}}}
\newcommand{\pRight}{\text{\underline{right}}}
\newcommand{\pState}{\text{\underline{state}}}
\newcommand{\pRecomb}{\text{\underline{recomb}}}
\newcommand{\pConcat}{\text{\underline{concat}}}
\newcommand{\pListconcat}{\text{\underline{listconcat}}}
\newcommand{\pTrans}{\text{\underline{trans}}}
\newcommand{\pDecomp}{\text{\underline{decomp}}}
\newcommand{\pHead}{\text{\underline{head}}}
\newcommand{\pTail}{\text{\underline{tail}}}
\newcommand{\pIntCopy}{\text{\underline{intcopy}}}
\newcommand{\pBintCopy}{\text{\underline{bintcopy}}}
\newcommand{\pNBoolCopy}{\text{${_n}$\underline{boolcopy}}}
\newcommand{\pStep}{\text{\underline{step}}}
\newcommand{\pRelStep}{\text{\underline{relstep}}}
\newcommand{\pBoolStep}{\text{\underline{boolstep}}}
\newcommand{\pExtract}{\text{\underline{extract}}}
\newcommand{\pPack}{\text{\underline{pack}}}
\newcommand{\pUnpack}{\text{\underline{unpack}}}
\newcommand{\pBoolstep}{\text{\underline{boolstep}}}
\newcommand{\pRead}{\text{\underline{read}}}
\newcommand{\pMultread}{\text{\underline{multread}}}
\newcommand{\pTensor}{\text{\underline{tensor}}}
\newcommand{\pComp}{\text{\underline{comp}}}
\newcommand{\pRepeat}{\text{\underline{repeat}}}
\newcommand{\pAdd}{\text{\underline{add}}}
\newcommand{\pMult}{\text{\underline{mult}}}
\newcommand{\pPred}{\text{\underline{pred}}}
\newcommand{\pIter}{\text{\underline{iter}}}
\newcommand{\pBooltype}{\text{\underline{booltype}}}
\newcommand{\pCast}{\text{\underline{cast}}}
\newcommand{\proj}{\operatorname{proj}}
\newcommand{\prob}{\bold{P}}
\newcommand{\probc}{\mathscr{P}}
\newcommand{\dkl}{D_{\operatorname{KL}}}

\newcommand{\nl}{\text{nl}} 
\newcommand{\dntn}[1]{\llbracket #1 \rrbracket} 
\newcommand{\dntntup}[1]{\langle\!\langle #1 \rangle\!\rangle} 
\newcommand{\dntnPT}[1]{\dntn{#1}_{\text{PT}}}
\newcommand{\dntnNL}[1]{\dntn{#1}_{\nl}}
\newcommand{\Sym}{\operatorname{Sym}}
\newcommand{\prom}{\operatorname{prom}}
\newcommand{\ND}{\text{ND}} 

\title{Encodings of Turing machines in Linear Logic}
\author{James Clift, Daniel Murfet}

\maketitle

\begin{abstract} We give several different encodings of the step function of a Turing machine in intuitionistic linear logic, and calculate the denotations of these encodings in the Sweedler semantics.
\end{abstract}

\tableofcontents

\section{Introduction}

Turing proved that his machines and terms of the lambda calculus can express the same class of functions by giving an encoding of the step function of a Turing machine as a lambda term, and simulating $\beta$-reduction with Turing machines \cite[Appendix]{turing3}. Although the two models of computation encode the same class of functions, there are nonetheless important aspects of computation that are more accessible in one model or the other. For example, \emph{time complexity} is immediately apparent in the Turing model, but much harder to locate in the theory of lambda terms. One ``native'' approach to time complexity on the side of lambda calculus was developed by Girard \cite{girard_complexity} in the context of second-order linear logic, using an encoding developed by him of Turing machines into this language. The present paper is concerned with an example running in the opposite direction: there is a theory of \emph{derivatives} for lambda terms developed by Ehrhard-Regnier \cite{difflambda} whose analogue we would like to understand on the side of Turing machines. 

We can of course take Turing's encoding of one of his machines $M$ as a lambda term and apply the Ehrhard-Regnier derivative to obtain a term $t_M$ of the differential lambda calculus. However, this answer remains on the side of lambda calculus, and \emph{a priori} does not have an obvious interpretation in terms of the operation of $M$. The question that we really want to answer is: \emph{what about $M$ does the term $t_M$ compute}? One natural approach is to use semantics in order to extract the desired content from $t_M$.
\vspace{0.2cm}

The present paper and its sequel \cite{clift_murfet3} take this approach using linear logic \cite{girard_llogic} rather than lambda calculus, primarily because the Sweedler semantics of linear logic in tensor algebra \cite{murfet_ll, clift_murfet} is well-suited for this purpose. In this paper we discuss various encodings of Turing machines into linear logic, including a modification of Girard's original encoding, and we calculate the denotations of these encodings. In \cite{clift_murfet3} we will analyse the denotations of the Ehrhard-Regnier derivatives of these encodings, and give our answer to the above question about the content of $t_M$.
\vspace{0.2cm}

\textbf{Outline of the paper.} The encoding of a Turing machine $M$ as a proof in linear logic is not unique, and indeed the paper is broadly organised around \emph{four} different variants, each with different tradeoffs:
\begin{itemize}
\item \emph{The Girard encoding} (Section \ref{section: turing machines}) encodes the state of the tape as a pair of binary integers, and is a modified form of the encoding in \cite{girard_complexity} with a more conservative use of second-order quantifiers. The step function of $M$ is encoded as a proof of
\be\label{eq:intro_girard}
{!} \tBint_{A^3} \otimes {!} \tBint_{A^3} \otimes {!} \tBool_{A^3} \vdash {!} \tBint_{A} \otimes {!} \tBint_{A} \otimes {!} \tBool_{A}
\ee
for any type $A$, where $\tBint_A$ denotes the type of binary integers and $\tBool_A$ denotes the type of booleans (see Section \ref{section:encoding_data_proofs}). The denotation of this encoding in the Sweedler semantics is given in Lemma \ref{lemma:denotation_step}.

\item \emph{The Boolean version of the Girard encoding} (Section \ref{section:actionofstep}) encodes the state of the tape as a sequence of booleans, and is defined by converting such a sequence into a pair of binary integers, running the Girard encoding for some number of steps, and then converting the resulting pair of binary integers representing the tape back into a sequence of booleans. The step function of $M$ is encoded by a family of proofs of
\be
a \,{!} \tBool_B \otimes b\, {!} \tBool_B \otimes {!} {_n} \tBool_B \vdash \big({!} \tBool_A\big)^{\otimes c } \otimes \big({!} \tBool_A\big)^{\otimes d} \otimes {!} {_n} \tBool_A
\ee
with $a,b,c,d \ge 1$ giving the bounds of the initial and final tape, and where $B$ is $A^g$ for some function $g$ of $c,d$ and the number of time steps.

\item \emph{The direct Boolean encoding} (Section \ref{section:direct_boolean}) also encodes the state of the tape using booleans and it comes in two flavours, one in which the tape contents are encoded \emph{relative} to the head position, and one in they are encoded in \emph{absolute} coordinates. In the relative case the step function of $M$ is encoded by proofs of
\be
{!}\tSBool_A^{\otimes 2h+1} \otimes {!}\tNBool_A \vdash {!}\tSBool_A^{\otimes 2h+3} \otimes {!}\tNBool_A
\ee
while in the absolute case it is encoded by proofs of
\be\label{eq:absstep_intro}
{!} \tSBool_A^{\otimes h} \otimes {!}\tNBool_A \otimes {!}\tHBool_A \vdash {!} \tSBool_A^{\otimes h} \otimes {!}\tNBool_A \otimes {!}\tHBool_A
\ee
where ${_s} \tBool_A$ is the type of $s$-booleans, used to encode tape symbols, and ${_h} \tBool_A$ is used to track the head position. The denotations of these encodings are given in Remark \ref{remark:denotation_relstep} and Remark \ref{remark:denotation_absstep} respectively.
\end{itemize}
These encodings all belong to a special class of proofs in linear logic which we call the \emph{component-wise plain proofs}. We introduce and give the basic properties of this class in Section \ref{section:plain_proofs}. Finally in Appendix \ref{section:ints_and_bints} we study the denotations of integers and binary integers under the Sweedler semantics, and prove some basic linear independence results that are needed in the main text (and which may be of independent interest).
\vspace{0.2cm}

Let us now briefly explain why we introduce four different encodings, and why Girard's original encoding in \cite{girard_complexity} is not suitable, in its original form, for our applications in differential linear logic. In Girard's original encoding he gives a proof which, when cut against an integer $n$ and an encoding of the initial configuration of the Turing machine, returns the configuration of the Turing machine after $n$ steps. This works by \emph{iterating} an encoding of the one-step transition function, and this iteration requires second-order quantifiers. From the point of view of derivatives the use of second-order is problematic, because it is not clear how derivatives in linear logic should interact with second-order. Our solution was to find a variation of Girard's encoding which introduces the second-order quantifiers only at the very bottom of the proof tree. This is the encoding given in Section \ref{section: turing machines}. In \cite{clift_murfet3} we study the Ehrhard-Regnier derivative of the first-order proof which stops just before the use of second-order quantifiers.

Differentiating proofs involves making infinitesimal variations in inputs, and in the case of a proof encoding the step function of a Turing machine, the most natural infinitesimal variations to make are those describing the contents of an \emph{individual} tape square. This is somewhat orthogonal to the approach of the Girard encoding, which uses two monolithic binary integers to encode the state of the tape. For this reason we were driven to develop the other encodings given above using sequence of booleans. 

\vspace{0.2cm}

\textbf{Related work.} The work of Roversi \cite{roversi} fixes an error in Girard's original encoding, but from our point of view this error is irrelevant, since it concerns whether or not Girard's encoding is typable in \emph{light} linear logic. A different encoding of Turing machines into linear logic is given by Mairson and Terui \cite[Theorem 5]{terui_complexity_cut_elim} which uses booleans based on tensors rather than additives.

\section{Background}

Throughout $k$ is an algebraically closed field, and all vector spaces and coalgebras are defined over $k$. Coalgebras are all coassociative, counital and cocommutative \cite{sweedler}. We write $\Prim(C)$ for the set of primitive elements in a coalgebra $C$.

\subsection{Linear logic and the Sweedler semantics}

As introductory references for linear logic we recommend \cite{mellies, benton_etal} and the survey \cite{murfet_ll}. In this paper \emph{linear logic} will always mean \emph{first-order intuitionistic linear logic} with connectives $\otimes, \&, \multimap, {!}$ and the corresponding introduction rules and cut-elimination transformations from \cite{mellies,benton_etal}. A \emph{proof} is always a proof of a sequent in linear logic. We refer to \cite{clift_murfet, murfet_ll} and \cite{hyland} for the definition of the \emph{Sweedler semantics} of linear logic in the category of vector spaces over $k$, which will be denoted throughout by $\dntn{-}$. Briefly, given formulas (we also use \emph{type} as a synonym for formula) $A,B$ the denotations are determined by the rules
\begin{align*}
\dntn{A \otimes B} &= \dntn{A} \otimes \dntn{B}\,\\
\dntn{A\, \&\, B} &= \dntn{A} \oplus \dntn{B}\,\\
\dntn{A \multimap B} &= \Hom_k(\dntn{A}, \dntn{B})\,\\
\dntn{{!} A} &= {!} \dntn{A}
\end{align*}
and a choice of vector spaces $\dntn{x}$ for atomic formulas $x$, where ${!} V$ denotes the universal cocommutative coassociative counital coalgebra mapping to $V$. The universal morphism is usually denoted $d_V: {!} V \lto V$ or just $d$. We review here the description of ${!} V$ and this universal map when $V$ is finite-dimensional; for the general case see \cite{murfet_coalg}, \cite[\S 5.2]{murfet_ll}.

If $V$ is finite-dimensional then
\[
{!} V = \bigoplus_{P \in V} \Sym_P(V)
\]
where $\Sym_P(V) = \Sym(V)$ is the symmetric coalgebra. If $e_1,\ldots,e_n$ is a basis for $V$ then as a vector space $\Sym(V) \cong k[e_1,\ldots,e_n]$. Given $v_1,\ldots,v_s \in V$, the corresponding tensor in $\Sym_P(V)$ is written using a ket
\[
\ket{v_1,\ldots,v_s}_P := v_1 \otimes \cdots \otimes v_s \in \Sym_P(V)\,.
\]
And in particular, the identity element of $\Sym_P(V)$ is denoted by a vacuum vector
\[
\vacu_P := 1 \in \Sym_P(V)\,.
\]
With this notation the universal map $d: {!} V \lto V$ is defined by
\[
d \vacu_P = P, \quad d\ket{v}_P = v, \quad d\ket{v_1,\ldots,v_s}_P = 0 \quad s > 1\,.
\]
The comultiplication on ${!} V$ is defined by
\[
\Delta \ket{ v_1, \ldots, v_s }_P = \sum_{I \subseteq \{ 1, \ldots, s \}} \ket{ v_I }_P \otimes \ket{ v_{I^c} }_P
\]
where $I$ ranges over all subsets including the empty set, and for a subset $I = \{ i_1, \ldots, i_p \}$ we denote by $v_I$ the sequence $v_{i_1},\ldots,v_{i_p}$, and $I^c$ is the complement of $I$. In particular
\[
\Delta \vacu_P = \vacu_P \otimes \vacu_P\,.
\]
The counit ${!} V \lto k$ is defined by $\vacu_P \mapsto 1$ and $\ket{v_1,\ldots,v_s}_P \mapsto 0$ for $s > 0$.

We write $A^n = A \;\& \;\cdots\; \&\; A$ where there are $n$ copies of $A$. Given a type $A$, $\pi : A$ means that $\pi$ is a proof of $\vdash A$. Given a proof $\pi$ of ${!} \Gamma \vdash B$ we write $\prom(\pi)$ for the proof obtained by applying the promotion rule to $\pi$ to obtain a proof of ${!} \Gamma \vdash {!} B$.

Whenever we talk about a set of proofs $\cat{P}$ of a formula $A$ in linear logic, we always mean a set of proofs \emph{modulo the equivalence relation of cut-elimination}. Given a set of proofs $\cat{N}$ we write $\dntn{\cat{N}}$ for $\{ \dntn{\nu} \}_{\nu \in \cat{N}}$. If proofs $\pi,\pi'$ are equivalent under cut-elimination then $\dntn{\pi} = \dntn{\pi'}$, so the function $\dntn{-}: \cat{P} \lto \dntn{A}$ extends uniquely to a $k$-linear map
\be\label{eq:denotesec}
\xymatrix@C+2pc{ k \cat{P} \ar[r]^-{\dntn{-}} & \dntn{A} }
\ee
where $k \cat{P}$ is the free $k$-vector space generated by the set $\cat{P}$. If $\psi$ is a proof of ${!} A_1,\ldots,{!} A_r \vdash B$ and $\alpha_i$ is a proof of $A_i$ for $1 \le i \le r$ then $\psi(\alpha_1,\ldots,\alpha_r) : B$ denotes the (cut-elimination equivalence class of) the proof obtained by cutting $\psi$ against the promotion of each $\alpha_i$.

\subsection{Encoding data as proofs}\label{section:encoding_data_proofs}

Given a base type $A$ we define
\begin{align*}
\tBool_A &= (A \;\&\; A)\multimap A\,,\\
{_n} \tBool_A &= A^n \multimap A\,,\\
\tInt_A &= {!}(A \multimap A) \multimap (A \multimap A)\,,\\
\tBint_A &= {!}(A \multimap A) \multimap \big({!}(A \multimap A) \multimap (A \multimap A)\big)\,.
\end{align*}
The encodings of booleans, integers and binary integers as proofs in linear logic go back to Girard's original paper \cite{girard_llogic}. For each integer $n \ge 0$ there is a corresponding proof $\underline{n}_A$ of $\tInt_A$ \cite[\S 3.1]{clift_murfet} and for $S \in \{0,1\}^*$ there is a corresponding proof $\underline{S}_A$ of $\tBint_A$ \cite[\S 3.2]{clift_murfet}.

The two values of a boolean correspond to the following proofs $\underline{0}_A$ and $\underline{1}_A$ of $\tBool_A$:
\begin{center}
\AxiomC{}
\UnaryInfC{$A \vdash A$}
\RightLabel{\scriptsize $\& L_0$}
\UnaryInfC{$A\,\&\,A \vdash A$}
\RightLabel{\scriptsize $\multimap R$}
\UnaryInfC{$\vdash \tBool_A$}
\DisplayProof
\qquad 
\AxiomC{}
\UnaryInfC{$A \vdash A$}
\RightLabel{\scriptsize $\& L_1$}
\UnaryInfC{$A\,\&\,A \vdash A$}
\RightLabel{\scriptsize $\multimap R$}
\UnaryInfC{$\vdash \tBool_A$}
\DisplayProof
\end{center}
whose denotations are projection onto the zeroth and first coordinates respectively. Note that we are using the convention that the left introduction rules for $\&$ are indexed by $0$ and $1$, rather than by the more conventional choice of $1$ and $2$, in order to be consistent with the usual assignment of $0$ as `false' and $1$ as `true'. 

The $n$ values of an $n$-boolean correspond to the projection maps $\proj_i: \dntn{A^n} \to \dntn{A}$, where $i \in \{0, ..., n-1\}$. We denote by $\underline{i}_A$ the proof
\begin{center}
\AxiomC{}
\UnaryInfC{$A \vdash A$}
\RightLabel{\scriptsize $\& L_i$}
\UnaryInfC{$A^n \vdash A$}
\RightLabel{\scriptsize $\multimap R$}
\UnaryInfC{$\vdash \tNBool_A$}
\DisplayProof
\end{center}
whose denotation is $\proj_i$. Here, by $\& L_{i}$ $(0 \leq i \leq n-1)$ we mean the rule which introduces $n-1$ new copies of $A$ on the left, such that the original $A$ is at position $i$, indexed from left to right.

\subsection{Cartesian products of coalgebras}\label{section:coalg_prod}

Let $(C, \Delta_C, \varepsilon_C)$ and $(D, \Delta_D, \varepsilon_D)$ be coalgebras. Then $C \otimes D$ is naturally a coalgebra \cite[p.49]{sweedler}, and it is the Cartesian product in the category of coalgebras \cite[p.65]{sweedler}. The following calculations are standard but will play such an important conceptual role in this paper and its sequel that it is worth reproducing them here.

\vspace{0.2cm}

Let $\Coalg_k$ denote the category of $k$-coalgebras.

\begin{lemma} Let $\pi_C: C \otimes D \lto C$ and $\pi_D: C \otimes D \lto D$ to be the composites
\be\label{eq:product_mapscoalg}
\xymatrix@C+2pc{ 
C \otimes D \ar[r]^-{1 \otimes \varepsilon_D} & C \otimes k \cong C
}\,,\qquad
\xymatrix@C+2pc{ 
C \otimes D \ar[r]^-{\varepsilon_C \otimes 1} & k \otimes D \cong D
}
\,.
\ee
These are coalgebra morphisms, and the tuple $(C \otimes D, \pi_C, \pi_D)$ is the Cartesian product of $C,D$ in the category of coalgebras.
\end{lemma}
\begin{proof}
The bijection
\be\label{eq:coalgebra_cartesian_iso}
\Coalg_k(X, C \otimes D) \lto \Coalg_k(X, C) \times \Coalg_k(X, D)
\ee
sends $\gamma$ to $(\pi_C \gamma, \pi_D \gamma)$ and its inverse sends a pair $(\alpha, \beta)$ to the composite $( \alpha \otimes \beta ) \circ \Delta_X$. The key point in showing that this is a bijection is the observation that any morphism of coalgebras $\gamma: X \lto C \otimes D$ may be reconstructed from its components $\pi_C \gamma, \pi_D \gamma$ by commutativity of the diagram
\[
\begin{tikzcd}
X \arrow[rr, "{\Delta_X}"]\arrow[dd, swap, "{\gamma}"]&& X \otimes X \arrow[d, "{\gamma \otimes \gamma}"]\arrow[rr, "{\pi_C \gamma \otimes \pi_D \gamma}"]&& C \otimes D 
\\
  && C \otimes D \otimes C \otimes D \arrow[rr, "{1 \otimes \varepsilon_D \otimes \varepsilon_C \otimes 1}"]&& C \otimes k \otimes k \otimes D \arrow{u}{\cong}
\\
C \otimes D \arrow[rrrrrrd, swap, "{1_{C \otimes D}}",bend right=6.5]\arrow[rr, "{\Delta_C \otimes \Delta_D}"]&& C \otimes C \otimes D \otimes D \arrow{u}{\cong}[swap]{T}\arrow[rr, "{1 \otimes \varepsilon_C \otimes \varepsilon_D \otimes 1}"]&& C \otimes k \otimes k \otimes D\arrow[drr, "\cong"] \arrow{u}{\cong}[swap]{T} \\
&&&&&& C \otimes D\arrow[uuull, swap, "{1_{C \otimes D}}",bend right=40]
\end{tikzcd}
\]
where $T$ is the twist map.
\end{proof}

In particular, the group-like \cite[p.57]{sweedler} elements decompose according to
\begin{align*}
G(C \otimes D) &\cong \Coalg_k( k, C \otimes D )\\
&\cong \Coalg_k(k, C) \times \Coalg_k(k, D)\\
&\cong G(C) \times G(D)\,.
\end{align*}
This isomorphism sends a pair $(c,d)$ of group-like elements to the group-like element $c \otimes d$ in $C \otimes D$. Similarly for the primitive elements
\begin{align*}\label{eq:primiso}
\Prim(C \otimes D) &\cong \Coalg_k( (k[\varepsilon]/\varepsilon^2)^*, C \otimes D )\\
&\cong \Coalg_k( (k[\varepsilon]/\varepsilon^2)^*, C) \times \Coalg_k( (k[\varepsilon]/\varepsilon^2)^*, D )\\
&\cong \Prim(C) \times \Prim(D)\,.
\end{align*}
A pair of primitive \cite[p.199]{sweedler} elements $x$ over $c$ in $C$ and $y$ over $d$ in $D$ correspond to a pair of morphisms of coalgebras
\begin{align*}
\widetilde{x}: (k[\varepsilon]/\varepsilon^2)^* \lto C\,, & \quad 1^* \mapsto c, \; \varepsilon^* \mapsto x\,\\
\widetilde{y}: (k[\varepsilon]/\varepsilon^2)^* \lto C\,, & \quad 1^* \mapsto d, \;\varepsilon^* \mapsto y\,.
\end{align*}
And the corresponding primitive element in $C \otimes D$ is the image of $\varepsilon^*$ under the map
\begin{gather*}
\xymatrix@C+2pc{
(k[\varepsilon]/\varepsilon^2)^* \ar[r]^-{\Delta} & (k[\varepsilon]/\varepsilon^2)^* \otimes (k[\varepsilon]/\varepsilon^2)^* \ar[r]^-{ \widetilde{x} \otimes \widetilde{y} } & C \otimes D
}\\
\varepsilon^* \longmapsto 1^* \otimes \varepsilon^* + \varepsilon^* \otimes 1^* \longmapsto c \otimes y + x \otimes d\,.
\end{gather*}


\begin{remark} A morphism of coalgebras sends primitive elements to primitive elements. In particular, if $\gamma: X \lto C \otimes D$ is a morphism of coalgebras then the restriction gives a map $\Prim(\gamma): \Prim(C) \lto \Prim(C \otimes D)$, and it is clear  that the diagram
\[
\xymatrix@C+5pc@R+1pc{
\Prim(X) \ar[r]^-{\Big(\begin{smallmatrix} \Prim(\gamma_C) \\ \Prim(\gamma_D) \end{smallmatrix}\Big)}\ar[dr]_-{\Prim(\gamma)} & \Prim(C) \times \Prim(D) \ar[d]^-{\cong}\\
& \Prim(C \otimes D)
}
\]
commutes, where the vertical map is the canonical isomorphism described above. Thus the action of $\gamma$ on primitive elements can be understood component-by-component.
\end{remark}

\section{Plain proofs}\label{section:plain_proofs}

In this section we introduce the class of \emph{plain proofs}. Our encodings of Turing machines in Section \ref{section: turing machines} will be component-wise plain. 

\begin{definition}\label{defn:plain} A proof of a sequent ${!} A_1,\ldots,{!} A_r \vdash B$ for $r \ge 0$ is \emph{plain} if it is equivalent under cut-elimination to
\begin{center}
\AxiomC{$\proofvdots{\pi}$}
\noLine\UnaryInfC{$n_1\, A_1,\ldots,n_r \, A_r \vdash B$}
\RightLabel{\scriptsize der}
\doubleLine\UnaryInfC{$n_1\, {!}A_1, \ldots, n_r \, {!} A_r \vdash B$}
\RightLabel{\scriptsize ctr/wk}
\doubleLine\UnaryInfC{${!}A_1,\ldots,{!}A_r \vdash B$}
\DisplayProof
\end{center}
for some proof $\pi$ and tuple of non-negative integers $\bold{n} = (n_1,\ldots,n_r)$, where for $n_i > 1$ in the final step there is a corresponding contraction, and if $n_i = 0$ the final step involves a weakening. We refer to the integer $n_i$ as the $A_i$-\emph{degree} and $\bold{n}$ as the \emph{degree vector}.
\end{definition}

\begin{example} Binary integers \cite[\S 3.2]{clift_murfet} give plain proofs of $2\,{!}(A \multimap A) \vdash A \multimap A$.
\end{example}

\begin{remark}\label{remark:stable_plain} If $\psi$ is a plain proof as above then:
\begin{itemize}
\item Suppose $A_i = A_j$ for some $i \neq j$ and let $\rho$ be the proof obtained from $\psi$ by contraction on ${!} A_i, {!} A_j$. Then $\rho$ is plain.
\item Let $\rho$ be the proof of ${!} A, {!} A_1, \ldots, {!} A_r \vdash B$ obtained from $\psi$ by weakening in the ${!} A$. Then $\rho$ is plain.
\end{itemize}
\end{remark}

\begin{lemma}\label{lemma:plain_under_comp} Suppose that $\theta_1, \ldots, \theta_r$ are plain proofs with conclusions $A_1,\ldots,A_r$ and that $\psi$ is a plain proof of ${!} A_1,\ldots,{!} A_r \vdash B$. Then the cut of $\psi$ against the promotions of the $\theta_i$ is a plain proof with conclusion $B$.
\end{lemma}
\begin{proof}
In the special case $r = 1$ the cut
\begin{center}
\AxiomC{$\proofvdots{\pi}$}
\noLine\UnaryInfC{$n\, A \vdash B$}
\RightLabel{\scriptsize$n\times$ der}
\doubleLine\UnaryInfC{$n\, {!}A \vdash B$}
\RightLabel{\scriptsize$n-1\times$ ctr}
\doubleLine\UnaryInfC{${!}A \vdash B$}
\RightLabel{\scriptsize prom}
\UnaryInfC{${!} A \vdash {!} B$}
\AxiomC{$\proofvdots{\rho}$}
\noLine\UnaryInfC{$m\, B \vdash C$}
\RightLabel{\scriptsize$m \times$ der}
\doubleLine\UnaryInfC{$m\, {!}B \vdash C$}
\RightLabel{\scriptsize$m-1\times$ ctr}
\doubleLine\UnaryInfC{${!}B \vdash C$}
\RightLabel{\scriptsize cut}
\BinaryInfC{${!} A \vdash C$}
\DisplayProof
\end{center}
is equivalent under cut elimination \cite[\S 3.9.3]{mellies} to a proof of the form
\begin{center}
\AxiomC{$\proofvdots{}$}
\noLine\UnaryInfC{$mn\, A \vdash C$}
\RightLabel{\scriptsize$mn\times$ der}
\doubleLine\UnaryInfC{$mn\, {!}A \vdash C$}
\RightLabel{\scriptsize$mn-1\times$ ctr}
\doubleLine\UnaryInfC{${!}A \vdash C$}
\DisplayProof
\end{center}
which is plain. The general case is similar.
\end{proof}

\begin{definition}\label{defn:plain_comp} A proof
\[
\psi: {!} A_1,\ldots,{!} A_r \vdash {!} B_1 \otimes \cdots \otimes {!} B_s
\]
is \emph{component-wise plain} if there are plain proofs
\[
\psi_i: {!} A_1, \ldots, {!} A_r \vdash B_i
\]
for $1 \le i \le s$ such that $\psi$ is equivalent under cut-elimination to the proof
\begin{center}
\AxiomC{$\proofvdots{\otimes_{i=1}^s \prom(\psi_i)}$}
\noLine\UnaryInfC{$s {!} A_1,\ldots, s {!} A_r \vdash {!} B_1 \otimes \cdots \otimes {!} B_s$}
\RightLabel{\scriptsize ctr}
\doubleLine\UnaryInfC{${!} A_1,\ldots,{!} A_r \vdash {!} B_1 \otimes \cdots \otimes {!} B_s$}
\DisplayProof 
\end{center}
We refer to the $\psi_i$ as the \emph{components} of $\psi$.
\end{definition}

Observe that in the context of the definition the linear map $\dntn{\prom(\psi_i)}$ is a morphism of coalgebras, and the denotation of the proof $\psi$ is precisely the morphism of coalgebras induced by the $\dntn{\prom(\psi_i)}$ into the tensor product of the ${!} \dntn{B_i}$ viewed as the Cartesian product in the category of coalgebras (Section \ref{section:coalg_prod}). At the syntactic level this means in particular that the components $\psi_i$ of a component-wise plain proof $\psi$ may be recovered (of course, up to cut-elimination) by cutting against a series of weakenings and a dereliction. 
\vspace{0.01cm}

This class of component-wise plain proofs is closed under composition:

\begin{proposition}\label{prop:cut_componentwiseplain} Suppose given two component-wise plain proofs
\begin{align*}
\psi: {!} A_1,\ldots,{!} A_r &\vdash {!} B_1 \otimes \cdots \otimes {!} B_s\,,\\
\phi: {!} B_1,\ldots,{!} B_s &\vdash {!} C_1 \otimes \cdots \otimes {!} C_t\,.
\end{align*}
Then the proof
\begin{center}
\AxiomC{$\proofvdots{\psi}$}
\noLine\UnaryInfC{${!} A_1,\ldots,{!} A_r \vdash {!} B_1 \otimes \cdots \otimes {!} B_s$}
\AxiomC{$\proofvdots{\phi}$}
\noLine\UnaryInfC{${!} B_1,\ldots,{!} B_s \vdash {!} C_1 \otimes \cdots \otimes {!} C_t$}
\RightLabel{\scriptsize $\otimes L$}
\doubleLine\UnaryInfC{${!} B_1 \otimes \cdots \otimes {!} B_s \vdash {!} C_1 \otimes \cdots \otimes {!} C_t$}
\RightLabel{\scriptsize cut}
\BinaryInfC{${!} A_1,\ldots,{!} A_r \vdash {!} C_1 \otimes \cdots \otimes {!} C_t$}
\DisplayProof 
\end{center}
which we denote $\phi \l \psi$, is component-wise plain.
\end{proposition}
\begin{proof}
By hypothesis $\psi, \phi$ are equivalent under cut-elimination to the tensor products of promotions of components $\psi_i, \phi_j$, and so $\phi \l \psi$ is equivalent to a proof
\begin{center}
\AxiomC{$\proofvdots{\otimes_i \prom(\psi_i)}$}
\noLine\UnaryInfC{$s {!} A_1,\ldots,s {!} A_r \vdash {!} B_1 \otimes \cdots \otimes {!} B_s$}
\RightLabel{\scriptsize ctr}
\doubleLine\UnaryInfC{${!} A_1,\ldots,{!} A_r \vdash {!} B_1 \otimes \cdots \otimes {!} B_s$}
\AxiomC{$\proofvdots{\otimes_j \prom(\phi_j)}$}
\noLine\UnaryInfC{$t {!} B_1,\ldots, t{!} B_s \vdash {!} C_1 \otimes \cdots \otimes {!} C_t$}
\RightLabel{\scriptsize ctr}
\doubleLine\UnaryInfC{${!} B_1, \ldots, {!} B_s \vdash {!} C_1 \otimes \cdots \otimes {!} C_t$}
\RightLabel{\scriptsize $\otimes L$}
\doubleLine\UnaryInfC{${!} B_1 \otimes \cdots \otimes {!} B_s \vdash {!} C_1 \otimes \cdots \otimes {!} C_t$}
\RightLabel{\scriptsize cut}
\BinaryInfC{${!} A_1,\ldots,{!} A_r \vdash {!} C_1 \otimes \cdots \otimes {!} C_t$}
\DisplayProof 
\end{center}
The cut-elimination rules of \cite[\S 3.10.2]{mellies} apply to transform this to a proof
\begin{center}
\AxiomC{$\proofvdots{\otimes_i \prom(\psi_i)}$}
\noLine\UnaryInfC{$s {!} A_1,\ldots,s {!} A_r \vdash {!} B_1 \otimes \cdots \otimes {!} B_s$}
\AxiomC{$\proofvdots{\otimes_j \prom(\phi_j)}$}
\noLine\UnaryInfC{$t {!} B_1,\ldots, t{!} B_s \vdash {!} C_1 \otimes \cdots \otimes {!} C_t$}
\RightLabel{\scriptsize ctr}
\doubleLine\UnaryInfC{${!} B_1, \ldots, {!} B_s \vdash {!} C_1 \otimes \cdots \otimes {!} C_t$}
\RightLabel{\scriptsize $\otimes L$}
\doubleLine\UnaryInfC{${!} B_1 \otimes \cdots \otimes {!} B_s \vdash {!} C_1 \otimes \cdots \otimes {!} C_t$}
\RightLabel{\scriptsize cut}
\BinaryInfC{$s{!} A_1,\ldots,s{!} A_r \vdash {!} C_1 \otimes \cdots \otimes {!} C_t$}
\RightLabel{\scriptsize ctr}
\doubleLine\UnaryInfC{${!} A_1,\ldots,{!} A_r \vdash {!} C_1 \otimes \cdots \otimes {!} C_t$}
\DisplayProof 
\end{center}
which by \cite[\S 3.8.1]{mellies} is equivalent under cut-elimination to cutting all the $\prom(\psi_i)$ against $\phi$ and then performing the contractions. But by \cite[\S 3.9.3]{mellies} the resulting proof is equivalent under cut-elimination to cutting $t$ copies of each $\prom(\psi_i)$ against the proof
\begin{center}
\AxiomC{$\proofvdots{\otimes_j \prom(\phi_j)}$}
\noLine\UnaryInfC{$t {!} B_1,\ldots, t{!} B_s \vdash {!} C_1 \otimes \cdots \otimes {!} C_t$}
\DisplayProof 
\end{center}
and then performing the contractions
\begin{center}
\AxiomC{$\proofvdots{}$}
\noLine\UnaryInfC{$st {!} A_1,\ldots, st{!} A_r \vdash {!} C_1 \otimes \cdots \otimes {!} C_t$}
\RightLabel{\scriptsize ctr}
\doubleLine\UnaryInfC{${!} A_1, \ldots, {!} A_r \vdash {!} C_1 \otimes \cdots \otimes {!} C_t$}
\DisplayProof 
\end{center}
Each of these cuts of $\prom(\psi_i)$ against $\otimes_j \prom(\phi_j)$ has as the final rule in the left branch a promotion and as the final rule in the right branch a right tensor introduction. The rules \cite[\S 3.11.1, \S 3.11.2]{mellies} transform the proof into a tensor product of sub-proofs $\kappa_j$ where a fixed $\prom(\phi_j)$ is cut against $\prom(\psi_1),\ldots,\prom(\psi_s)$ to obtain a proof $\kappa_j: s {!} A_1, \ldots, s{!} A_r \vdash {!} C_j$. Then $\kappa_1 \otimes \cdots \otimes \kappa_t$ is subject to the contractions as above. 

Using the (\emph{Promotion}, \emph{Promotion})-rule\footnote{This rule appears to be missing from the cut-elimination transformations in \cite{mellies}.} \cite[\S 5.2]{benton_etal} the proof $\kappa_j$ is equivalent under cut-elimination to the promotion of the proof $\kappa'_j$ which results from cutting of all the $\prom(\psi_i)$ against $\phi_j$. This $\kappa'_j$ is by Lemma \ref{lemma:plain_under_comp} a plain proof. We have now shown that $\phi \l \psi$ is equivalent under cut-elimination to the tensor product over $1 \le j \le t$ of promotions of plain proofs $\kappa'_j: s {!} A_1, \ldots, s {!} A_r \vdash C_j$ followed by the contractions above. Hence $\phi \l \psi$ is component-wise plain.
\end{proof}

\subsection{Denotations}\label{section:denotations_plain}

For the rest of this section suppose given a plain proof
\[
\psi : {!} A_1,\ldots,{!} A_r \vdash B
\]
constructed from $\pi$ as in Definition \ref{defn:plain}. The denotation of $\psi$ is a linear map
\[
\dntn{\psi}: {!} \dntn{A_1} \otimes \cdots \otimes {!} \dntn{A_r} \lto \dntn{B}\,.
\]
Suppose given finite sets of proofs $\cat{P}_i$ of $A_i$ and $\cat{Q}$ of $B$, such that
\be\label{eq:coll_comput}
\Big\{ \pi( X_1,\ldots,X_r ) \l X_i \in \cat{P}_i^{n_i} \Big\} \subseteq \cat{Q}
\ee
where the $n_i$ are as in Definition \ref{defn:plain}. Given a set of proofs $\cat{N}$ we write $\dntn{\cat{N}}$ for $\{ \dntn{\nu} \}_{\nu \in \cat{N}}$. We assume throughout that $\{ \dntn{\nu} \}_{\nu \in \cat{Q}}$ is linearly independent in $\dntn{B}$. In our examples $B$ is one of the standard datatypes and Appendix \ref{section:ints_and_bints} gives us a supply of linearly independent proof denotations. Let us first of all examine the polynomials that arise in evaluating $\dntn{\pi}$. 

Given a function $\gamma: \{1,\ldots,n_i\} \lto \cat{P}_i$ for some $i$, we write
\begin{align}
\rho_\gamma &= \gamma(1) \otimes \cdots \otimes \gamma(n_i) : A_i^{\otimes n_i} \label{eq:rhogamma}\\
\dntn{\rho_\gamma} &= \dntn{ \gamma(1) } \otimes \cdots \otimes \dntn{\gamma(n_i)} \in \dntn{A_i}^{\otimes n_i}\,.
\end{align}
Observe that for $\lambda^{ij}_\rho \in k$, and with $\gamma_i$ ranging over all functions $\{1,\ldots,n_i\} \lto \cat{P}_i$,
\begin{align*}
\dntn{\pi}\Big( \bigotimes_{i=1}^r \bigotimes_{j=1}^{n_i} \sum_{\rho \in \cat{P}_i} \lambda^{ij}_{\rho} \dntn{\rho} \Big) &= \dntn{\pi}\Big( \bigotimes_{i=1}^r \sum_{\gamma_i} \Big\{ \prod_{j=1}^{n_i} \lambda^{ij}_{\gamma_i(j)} \Big\} \dntn{ \rho_{\gamma_i}} \Big)\\
&= \sum_{\gamma_1,\ldots,\gamma_r} \Big\{ \prod_{i=1}^r \prod_{j=1}^{n_i} \lambda^{ij}_{\gamma_i(j)} \Big\} \dntn{\pi}\big( \dntn{\rho_{\gamma_1}} \otimes \cdots \otimes \dntn{\rho_{\gamma_r}} \big)\\
&= \sum_{\gamma_1,\ldots,\gamma_r} \Big\{ \prod_{i=1}^r \prod_{j=1}^{n_i} \lambda^{ij}_{\gamma_i(j)} \Big\} \dntn{\pi\big( \rho_{\gamma_1}, \ldots, \rho_{\gamma_r}\big)}\\
&= \sum_{\tau \in \cat{Q}} \Big\{ \sum_{\gamma_1,\ldots,\gamma_r} \delta_{\tau = \pi(\rho_{\gamma_1},\ldots,\rho_{\gamma_r})} \prod_{i=1}^r \prod_{j=1}^{n_i} \lambda^{ij}_{\gamma_i(j)} \Big\} \dntn{\tau}\,.
\end{align*}
Let us introduce variables $\{ x^{ij}_\rho \}_{1 \le i \le r, 1\le j \le n_i, \rho \in \cat{P}_i}$ so that with
\begin{align*}
\Sym(k\cat{P}_1^{n_1} \oplus \cdots \oplus k\cat{P}_r^{n_r} ) &\cong k\Big[ \{ x^{ij}_\rho \}_{i,j,\rho} \Big]
\end{align*}
we may define an element of this algebra by:

\begin{definition} Given $\tau \in \cat{Q}$ set
\[
g_\pi^\tau := \sum_{\gamma_1,\ldots,\gamma_r} \delta_{\tau = \pi(\rho_{\gamma_1},\ldots,\rho_{\gamma_r})} \prod_{i=1}^r \prod_{j=1}^{n_i} x^{ij}_{\gamma_i(j)}\,.
\]
\end{definition}
In summary, we may compute $\dntn{\pi}$ by these polynomials using the formula
\be\label{eq:calculationdntnpi}
\dntn{\pi}\Big( \bigotimes_{i=1}^r \bigotimes_{j=1}^{n_i} \sum_{\rho \in \cat{P}_i} \lambda^{ij}_{\rho} \dntn{\rho} \Big) = \sum_{\tau \in \cat{Q}} g_\pi^\tau \Big\vert_{x^{ij}_\rho = \lambda^{ij}_\rho} \dntn{\tau}\,.
\ee
Now let us turn to calculating $\dntn{\psi}$ using the morphism of $k$-algebras
\begin{gather}
C: \Sym\big( \bigoplus_{i=1}^r k\cat{P}_i^{n_i} \big) \cong k\Big[ \{ x^{ij}_\rho \}_{i,j,\rho} \Big] \lto k\Big[ \{ x^{i}_\rho \}_{i,j,\rho} \Big] \cong \Sym\big( \bigoplus_{i=1}^r k\cat{P}_i \big) \label{eq:Cmap}\\
x^{ij}_\rho \mapsto x^i_{\rho}\,.\nonumber
\end{gather}
With $\omega_i = \sum_{\rho \in \cat{P}_i} \lambda^i_{\rho} \dntn{\rho}$ we calculate
\begin{align}
\dntn{\psi}\Big( \vacu_{\omega_1} \otimes \cdots \otimes \vacu_{\omega_r} \Big) &= \dntn{\pi}\Big( \bigotimes_{i=1}^r d^{\otimes n_i} \Delta^{n_i-1} \vacu_{\omega_i} \Big) \nonumber\\
&= \dntn{\pi}\Big( \bigotimes_{i=1}^r d^{\otimes n_i} \vacu_{\omega_i}^{\otimes n_i} \Big)\nonumber\\
&= \dntn{\pi}\Big( \bigotimes_{i=1}^r \omega_i^{\otimes n_i} \Big)\nonumber\\
&= \sum_{\tau \in \cat{Q}} C( g^\tau_\pi )\Big\vert_{x^i_\rho = \lambda^i_\rho} \dntn{\tau}\,. \label{eq:pidaggercalc}
\end{align}
Let $\iota$ denote the function
\begin{gather*}
\iota: \prod_{i=1}^r k\cat{P}_i \lto \bigotimes_{i=1}^r {!} \dntn{A_i}\,,\\
\iota\big( \omega_1, \ldots, \omega_r \big) = \bigotimes_{i=1}^r \vacu_{\dntn{\omega_i}}
\end{gather*}
where $k \cat{P}$ is the free vector space on $\cat{P}$.

\begin{proposition}\label{prop:fpi} There is a unique function $F_\psi$ making the diagram
\[
\xymatrix@C+3pc@R+1pc{
{!} \dntn{A_1} \otimes \cdots \otimes {!} \dntn{A_r} \ar[r]^-{\dntn{\psi}} & \dntn{B}\\
k\cat{P}_1 \times \cdots \times k\cat{P}_r \ar[u]^-{\iota}\ar[r]_-{F_\psi} & k\cat{Q} \ar[u]_-{\dntn{-}}
}
\]
commute. Moreover this function is computed by a polynomial, in the sense that it is induced by a morphism of $k$-algebras
\[
f_\psi : \Sym(k \cat{Q}) \lto \Sym( k\cat{P}_1 \oplus \cdots \oplus k \cat{P}_r )\,.
\]
More precisely, if we present the symmetric algebras as polynomial rings in variables
\[
\{ y_\tau \}_{\tau \in \cat{Q}}\,,\qquad \{ x^i_\rho \}_{1 \le i \le r, \rho \in \cat{P}_i}
\]
respectively, then the polynomial $f^\tau_\psi := f_\psi( y_\tau )$ is given by the formula
\be\label{eq:formulaforfpitau}
f_\psi^\tau = \sum_{\gamma_1,\ldots,\gamma_r} \delta_{\tau = \pi(\rho_{\gamma_1},\ldots,\rho_{\gamma_r})} \prod_{i=1}^r \prod_{j=1}^{n_i} x^i_{\gamma_i(j)}\,,
\ee
where $\gamma_i$ ranges over all functions $\{1,\ldots,n_i\} \lto \cat{P}_i$ and we use the notation of \eqref{eq:rhogamma}.
\end{proposition}
\begin{proof}
Since by hypothesis $\{ \dntn{\nu} \}_{\nu \in \cat{Q}}$ is linearly independent the right hand vertical map (from \eqref{eq:denotesec}) is injective and so the map $F_{\psi}$ is unique if it exists. Existence follows from \eqref{eq:pidaggercalc}, and moreover this also shows that $F_\psi(\bold{a})_\tau = \operatorname{eval}_{x^i_\rho = a^i_\rho}( f_\pi^\tau )$ for all $\bold{a}$ in $\prod_i k \cat{P}_i$.
\end{proof}


\section{The Girard encoding} \label{section: turing machines}

To fix notation we briefly recall the definition of a Turing machine from \cite{arorabarak, sipser_intro_to_theory_of_computation}. Informally speaking, a Turing machine is a computer which possesses a finite number of internal states, and a one dimensional `tape' as memory. We adopt the convention that the tape is unbounded in both directions. The tape is divided into individual squares each of which contains some symbol from a fixed alphabet; at any instant only one square is being read by the `tape head'. Depending on the symbol on this square and the current internal state, the machine will write a symbol to the square under the tape head, possibly change the internal state, and then move the tape head either left or right. Formally,

\begin{definition} \label{defn: turing machine}
A \emph{Turing machine} $M = (\Sigma, Q, \delta)$ is a tuple where $Q$ is a finite set of states, $\Sigma$ is a finite set of symbols called the \emph{tape alphabet}, and \[\delta: \Sigma \times Q \to \Sigma \times Q \times \{\text{left, right}\}\] is a function, called the \emph{transition function}.
\end{definition}

The set $\Sigma$ is assumed to contain some designated blank symbol $\Box$ which is the only symbol that is allowed to occur infinitely often on the tape. Often one also designates a starting state, as well as a special accept state which terminates computation if reached.

If $M$ is a Turing machine, a \emph{Turing configuration} of $M$ is a tuple $\<S, T, q\>$, where $S, T \in \Sigma^*$ and $q \in Q$. This is interpreted as the instantaneous configuration of the Turing machine in the following way. The string $S$ corresponds to the non-blank contents of the tape to the left of the tape head, including the symbol currently being scanned. The string $T$ corresponds to a \textit{reversed} copy of the contents of the tape to the right of the tape head, and $q$ stores the current state of the machine. The reason for $T$ being reversed is a matter of convenience, as we will see in the next section. The \emph{step function}
\[
{}_\delta \textrm{step}: \Sigma^* \times \Sigma^* \times Q \lto \Sigma^* \times \Sigma^* \times Q
\]
sends the current configuration $\<S,T,q\>$ to the configuration ${_\delta} \textrm{step}( S, T, q )$ after one step.

The eventual goal of this section will be to present a method of encoding of Turing machines in linear logic. This is heavily based on work by Girard in \cite{girard_complexity}, which encodes Turing configurations via a variant of second order linear logic called \textit{light} linear logic. The encoding does not use light linear logic in a crucial way, but requires second order in many intermediate steps, making it incompatible with differentiation. We modify this encoding so that it is able to be differentiated (see Remark \ref{remark:second-order-final}), while also filling in some of the details omitted from \cite{girard_complexity}.

\begin{definition} \label{defn: turing configs}
Fix a finite set of states $Q = \{0, ..., n-1\}$, and a tape alphabet\footnote{It is straightforward to modify what follows to allow larger tape alphabets, see Appendix \ref{section:tape_alphabet}. At any rate, any Turing machine can be simulated by one whose tape alphabet is $\{0,1\}$, so this really isn't as restrictive as it might seem \cite{rogozhin_small_univ_turing_machines}.} $\Sigma = \{0,1\}$, with $0$ being the blank symbol. The type of \emph{Turing configurations} on $A$ is:
\[
\tTur_A = {!}\tBint_A \otimes {!}\tBint_A \otimes {!}\tNBool_A.
\]
The configuration $\<S, T, q\>$ is represented by the element 
\[
\dntn{\<S, T, q\>} = \vacu_{\dntn{\underline{S}_A}} \otimes \vacu_{\dntn{\underline{T}_A}}  \otimes \vacu_{\dntn{\underline{q}_A}}  \in \dntn{\tTur_A}.
\]
\end{definition}

Our aim is to simulate a single transition step of a given Turing machine $M$ as a proof ${_\delta}\pStep_A$ of $\tTur_B \vdash \tTur_A$ for some formula $B$ which depends on $A$, in the sense that if said proof is cut against a Turing configuration of $M$ at time $t$, the result will be equivalent under cut elimination to the Turing configuration of $M$ at time step $t+1$. This will be achieved in Theorem \ref{thm: single transition step}. Inspired by \cite{girard_complexity} our strategy will be as follows. Let $\<S\sigma, T\tau, q\>$ be the (initial) configuration of the given Turing machine.
\begin{enumerate}
\item Decompose the binary integers $S \sigma$ and $T \tau$ to extract their final digits, giving $S, T, \sigma$ and $\tau$. Note that $\sigma$ is the symbol currently under the head, and $\tau$ is the symbol immediately to its right.
\item Using the symbol $\sigma$ together with the current state $q \in Q$, compute the new symbol $\sigma'$, the new state $q'$, and the direction to move $d$.
\item If $d = \text{right}$, append $\sigma' \tau$ to $S$. Otherwise, append $\tau \sigma'$ to $T$; remember that the binary integer $T$ is the \textit{reversal} of the contents of the tape to the right of the tape head. This is summarised in Figure \ref{fig: single transition step}.
\end{enumerate}

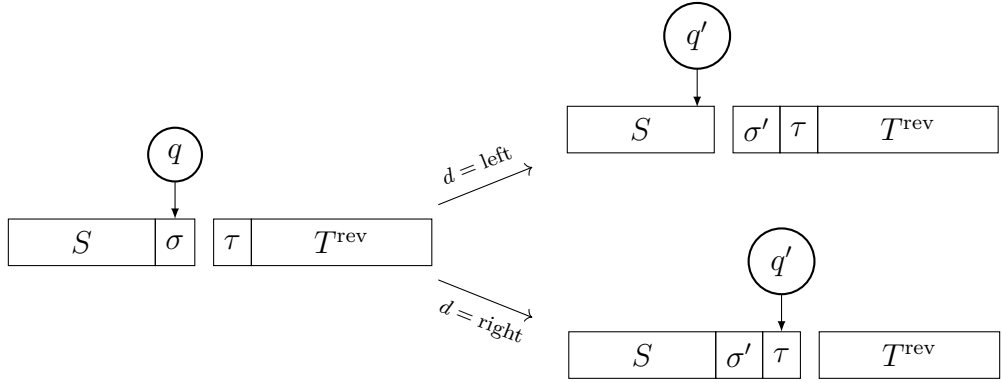
\begin{figure}
\centering
\begin{tikzpicture}[every node/.style={block},
        block/.style={minimum height=1.5em,outer sep=0pt,draw,rectangle,node distance=0pt}]
    \node (A) {$\sigma$};
    \node (B) [left=of A] {$\hspace{0.7cm} S \hspace{0.7cm}$};
    \node (C) [right= 0.25cm of A] {$\tau$};
    \node (D) [right=of C] {$\hspace{0.7cm} T^\text{rev} \hspace{0.7cm}$};
    \node (E) [above = 0.5cm of A,draw=black,thick, circle] {$q$};
    \draw[-latex] (E) -- (A); 
    \node  at (7.73,1.5) (F){$\sigma'$};
    \node (G) [left= 0.25cm of F] {$\hspace{0.7cm} S \hspace{0.7cm}$};
    \node (H) [right=of F] {$\tau$};
    \node (I) [right=of H] {$\hspace{0.7cm} T^\text{rev}  \hspace{0.7cm}$};
    \node (J) [above = 0.5cm of G, xshift=0.75cm,draw=black,thick, circle] {$q'$};
    \draw[-latex] (J) -- ++(0,-0.95);
    \node  at (7.5,-1.5) (K){$\sigma'$};
    \node (L) [left=of K] {$\hspace{0.7cm} S \hspace{0.7cm}$};
    \node (M) [right=of K] {$\tau$};
    \node (N) [right= 0.25cm of M] {$\hspace{0.7cm} T^\text{rev}  \hspace{0.7cm}$};
    \node (O) [above = 0.5cm of M,draw=black,thick, circle] {$q'$};
    \draw[-latex] (O) -- (M); 
    \draw[->] (3.5,0.5) -- (4.75,1) node [sloped, pos=0.5,above, draw=none] {\scriptsize $d = \text{left}$}; 
    \draw[->] (3.5,-0.5) -- (4.75,-1) node [sloped, pos=0.55, below, draw=none] {\scriptsize $d = \text{right}$};
\end{tikzpicture}    
    
\caption{A single transition step of a Turing machine.}\label{fig: single transition step}
\end{figure}

For simplicity in the main body of the text we present the encoding where the head of the Turing machine must either move left or right at each time step, but we explain in Appendix \ref{section:headstill} the minor changes necessary to allow the head to also remain stationary.

\subsection{The encoding} \label{subsection: the encoding}

In order to feed the current symbol into the transition function, it is necessary to extract this digit from the binary integer which represents the tape. To do this we must decompose a binary integer $S' = S\sigma$ of length $l \geq 1$ into two parts $S$ and $\sigma$, the former being a $\tBint$ consisting of the first $l-1$ digits (the \textit{tail}) and the latter being a $\tBool$ corresponding to the final digit (the \textit{head}). 

\begin{proposition} \label{decomposing bints head}
There exists a proof $\pHead_A$ of $\tBint_{A^3} \vdash \tBool_A$ which encodes the function $\dntn{\underline{S \sigma}_{A^3}} \mapsto \dntn{\underline{\sigma}_A}$.
\end{proposition}

\begin{proof}
The construction we will use is similar to that in \cite[\S 2.5.3]{girard_complexity}. Let $\pi_0, \pi_1$ be the (easily constructed) proofs of $A^3 \vdash A^3$ whose denotations are $\dntn{\pi_0}(x,y,z) = (x,y,x)$ and $\dntn{\pi_1}(x,y,z) = (x,y,y)$ respectively. Similarly let $\rho$ be the proof of $A^2 \vdash A^3$ with denotation $\dntn{\rho}(x,y) = (x,y,x)$. Define by $\pHead_A$ the following proof:
\begin{center}
\AxiomC{$\proofvdots{\pi_{0}}$}
\noLine\UnaryInfC{$A^3 \vdash A^3$}
\RightLabel{\scriptsize $\multimap R$}
\UnaryInfC{$\vdash A^3 \multimap A^3$}
\RightLabel{\scriptsize prom}
\UnaryInfC{$\vdash {!}(A^3 \multimap A^3)$}

\AxiomC{$\proofvdots{\pi_{1}}$}
\noLine\UnaryInfC{$A^3 \vdash A^3$}
\RightLabel{\scriptsize $\multimap R$}
\UnaryInfC{$\vdash A^3 \multimap A^3$}
\RightLabel{\scriptsize prom}
\UnaryInfC{$\vdash {!}(A^3 \multimap A^3)$}

\AxiomC{$\proofvdots{\rho}$}
\noLine\UnaryInfC{$A^2 \vdash A^3$}

\AxiomC{}
\UnaryInfC{$A \vdash A$}
\RightLabel{\scriptsize $\& L_2$}
\UnaryInfC{$A^3 \vdash A$}

\RightLabel{\scriptsize $\multimap L$}
\BinaryInfC{$A^3 \multimap A^3, A^2 \vdash A$}
\RightLabel{\scriptsize $\multimap R$}
\UnaryInfC{$A^3 \multimap A^3 \vdash \tBool_A$}
\RightLabel{\scriptsize $\multimap L$}
\BinaryInfC{$\tInt_{A^3} \vdash \tBool_A$}
\RightLabel{\scriptsize $\multimap L$}
\BinaryInfC{$\tBint_{A^3} \vdash \tBool_A$}
\DisplayProof
\end{center}
where the rule $\&L_2$ introduces two new copies of $A$ on the left, such that the original copy is at position $2$ (that is, the third element of the triple). 

We now show that $\dntn{\pHead_A}(\dntn{\underline{S\sigma}_{A^3}}) = \dntn{\underline{\sigma}_A}$ as claimed. Recall that the denotation $\dntn{\underline{S\sigma}_{A^3}}$ of a binary integer is a function which, given inputs $\alpha$ and $\beta$ of type $A^3 \multimap A^3$ corresponding to the digits zero and one, returns some composite of $\alpha$ and $\beta$. The effect of the two leftmost branches of $\pHead_A$ is to substitute $\dntn{\pi_0}$ for $\alpha$ and $\dntn{\pi_1}$ for $\beta$ in this composite, giving a linear map $\varphi: \dntn{A^3} \to \dntn{A^3}$. The rightmost branch then computes $\proj_2 \circ \,\varphi \circ \dntn{\rho}: \dntn{A^2} \to \dntn{A}$, giving a boolean.

In other words, $\dntn{\pHead_A}(\dntn{\underline{S\sigma}_{A^3}})$ is the element of $\dntn{\tBool_A}$ given by:
\[
(a_0, a_1) \mapsto \proj_2\circ \,\varphi\circ \dntn{\rho}(a_0, a_1) = \proj_2\circ \,\varphi(a_0, a_1, a_0),
\]
where $\varphi$ is the composite of $\dntn{\pi_0}$ and $\dntn{\pi_1}$ as above. Note however that repeated applications of the functions $\dntn{\pi_i}$ only serve to update the final digit of the triple, and thus only the final copy of $\dntn{\pi_i}$ determines the output value. Hence the above simplifies to \[\proj_2 \circ \,\varphi (a_0,a_1, a_0) = \proj_2 \circ \dntn{\pi_\sigma} (a_0, a_1, a_0) = \proj_2(a_0, a_1, a_\sigma) = a_\sigma.\]
Thus $\dntn{\pHead_A}(\dntn{\underline{S\sigma}_{A^3}}) = \proj_{\sigma}$, which is indeed the boolean corresponding to $\sigma$.

Lastly, we consider the special case when $S\sigma$ is the empty list. In this case
\[
\proj_2 \dntn{\rho}(a_0, a_1) = \proj_2 (a_0, a_1, a_0) = a_0
\]
which captures the fact that any symbols outside the working section of the tape are assumed to be the blank symbol, $0$.
\end{proof}

\begin{proposition} \label{decomposing bints tail}
There exists a proof $\pTail_A$ of $\tBint_{A^3} \vdash \tBint_A$ which encodes the function $\dntn{\underline{S \sigma}_{A^3}} \mapsto \dntn{\underline{S}_A}$.
\end{proposition}

\begin{remark}
This could also be encoded as a proof of $\tBint_{A^2} \vdash \tBint_A$. However it will be much more convenient later if the sequents proven by $\pHead_A$ and $\pTail_A$ have the same premise, since we will need to apply them both to two copies of the same binary integer.
\end{remark}

\begin{proof}
This is largely based on the predecessor for $\tInt_A$. Define $\pi$ to be the proof
\begin{center}
\AxiomC{}
\UnaryInfC{$A \vdash A$}
\AxiomC{}
\UnaryInfC{$A \vdash A$}
\AxiomC{}
\UnaryInfC{$A \vdash A$}
\doubleLine\RightLabel{\scriptsize $\& R$}
\TrinaryInfC{$A \vdash A^3$}
\AxiomC{}
\UnaryInfC{$A \vdash A$}
\RightLabel{\scriptsize $\& L_0$}
\UnaryInfC{$A^3 \vdash A$}
\RightLabel{\scriptsize $\multimap L$}
\BinaryInfC{$A, A^3 \multimap A^3 \vdash A$}
\RightLabel{\scriptsize $\multimap R$}
\UnaryInfC{$A^3 \multimap A^3 \vdash A \multimap A$}
\DisplayProof
\end{center}
which has denotation:
\[
\dntn{\pi}(\varphi)(a) = \proj_0(\varphi(a, a, a)).
\]
Define $\rho$ to be the following proof:
\begin{center}
\AxiomC{}
\UnaryInfC{$A \vdash A$}
\RightLabel{\scriptsize $\&L_2$}
\UnaryInfC{$A^3 \vdash A$}
\RightLabel{\scriptsize weak}
\UnaryInfC{$A^3, {!}(A \multimap A) \vdash A$}

\AxiomC{}
\UnaryInfC{$A \vdash A$}
\RightLabel{\scriptsize $\&L_2$}
\UnaryInfC{$A^3 \vdash A$}
\RightLabel{\scriptsize weak}
\UnaryInfC{$A^3, {!}(A \multimap A) \vdash A$}

\AxiomC{}
\UnaryInfC{$A \vdash A$}
\AxiomC{}
\UnaryInfC{$A \vdash A$}
\RightLabel{\scriptsize $\multimap L$}
\BinaryInfC{$A, A \multimap A \vdash A$}
\RightLabel{\scriptsize $\& L_2$}
\UnaryInfC{$A^3, A \multimap A \vdash A$}
\RightLabel{\scriptsize der}
\UnaryInfC{$A^3, {!}(A \multimap A) \vdash A$}

\doubleLine\RightLabel{\scriptsize $\& R$}
\TrinaryInfC{$A^3, {!}(A \multimap A) \vdash A^3$}
\RightLabel{\scriptsize $\multimap R$}
\UnaryInfC{${!}(A \multimap A) \vdash A^3 \multimap A^3$}
\DisplayProof
\end{center}
The denotation $\dntn{\rho}$ is
\[
\dntn{\rho}(\ket{\alpha_1, ..., \alpha_s}_\gamma)(a_0, a_1, a_2) = \begin{cases}(a_2, a_2, \gamma a_2) & s = 0 \\ (a_2, a_2, \alpha_1 a_2) & s = 1 \\ (a_2, a_2, 0) & s > 1.\end{cases}
\]
Finally, define $\pTail_A$ to be the following proof:
\begin{center}
\small
\AxiomC{$\proofvdots{\rho}$}
\noLine\UnaryInfC{${!}(A \multimap A) \vdash A^3 \multimap A^3$}
\RightLabel{\scriptsize prom}
\UnaryInfC{${!}(A \multimap A) \vdash {!}(A^3 \multimap A^3)$}

\AxiomC{$\proofvdots{\rho}$}
\noLine\UnaryInfC{${!}(A \multimap A) \vdash A^3 \multimap A^3$}
\RightLabel{\scriptsize prom}
\UnaryInfC{${!}(A \multimap A) \vdash {!}(A^3 \multimap A^3)$}

\AxiomC{$\proofvdots{\pi}$}
\noLine\UnaryInfC{$A^3 \multimap A^3 \vdash A \multimap A$}

\RightLabel{\scriptsize $\multimap L$}
\BinaryInfC{${!}(A \multimap A), \tInt_{A^3} \vdash A \multimap A$}
\RightLabel{\scriptsize $\multimap L$}
\BinaryInfC{${!}(A \multimap A), {!}(A \multimap A), \tBint_{A^3} \vdash A \multimap A$}
\doubleLine\RightLabel{\scriptsize $2 \times \multimap R$}
\UnaryInfC{$\tBint_{A^3} \vdash \tBint_A$}
\DisplayProof
\end{center}
Evaluated on the binary integer $S$, this gives a binary integer $T$ which if fed two vacuum vectors $\vacu_\gamma$ and $\vacu_\delta$ (corresponding to the digits 0, 1) will return the composite $\dntn{A} \to \dntn{A}$ obtained by substituting $\dntn{\rho} \vacu_\gamma$ and $\dntn{\rho} \vacu_\delta$ for each copy of the digits 0 and 1 respectively in $S$, and then finally keeping the $0$th projection by the left introduction of $\pi$.

As an example, suppose that the binary integer $S$ is $0010$. Then the corresponding linear map $\dntn{A} \to \dntn{A}$ is
\[
a \mapsto \proj_0(\tilde{\gamma}\circ \tilde{\delta} \circ \tilde{\gamma} \circ\tilde{\gamma}(a,a,a))
\]
where $\tilde{\gamma} = \dntn{\rho} \vacu_\gamma$, which is the morphism $(a_0,a_1,a_2) \mapsto (a_2, a_2, \gamma a_2)$, and similarly for $\tilde{\delta}$. Thus, we have:
\begin{align*}
\proj_0(\tilde{\gamma}\circ \tilde{\delta} \circ \tilde{\gamma} \circ\tilde{\gamma}(a,a,a)) 
&= \proj_0(\tilde{\gamma}\circ \tilde{\delta} \circ \tilde{\gamma}(a, a, \gamma(a)))
\\&= \proj_0(\tilde{\gamma}\circ \tilde{\delta}(\gamma(a), \gamma(a), \gamma\gamma(a)))
\\&= \proj_0(\tilde{\gamma}(\gamma\gamma(a),\gamma\gamma(a), \delta\gamma\gamma(a)))
\\&= \proj_0(\delta\gamma\gamma(a),\delta\gamma\gamma(a), \gamma\delta\gamma\gamma(a))
\\&= \delta\gamma\gamma(a).
\end{align*}
\end{proof}

When fed through the decomposition steps, the base type of the binary integers changes from $A^3$ to $A$. We therefore also need to modify the base type of the $n$-boolean representing the state, in order to keep the base types compatible.

\begin{lemma} \label{fixing_bool_types}
There exists a proof ${_n}\pBooltype_A$ of $\tNBool_{A^3} \vdash \tNBool_A$ which converts an $n$-boolean on $A^3$ to the equivalent $n$-boolean on $A$; that is, it encodes $\dntn{\underline{i}_{A^3}} \mapsto \dntn{\underline{i}_A}$.
\end{lemma}

\begin{proof}
For $i \in \{0, ..., n-1\}$, let $\pi_i$ be the proof of $A^n \vdash A^3$ whose denotation is $(a_0, ..., a_{n-1}) \mapsto (a_i, a_i, a_i)$. Define ${_n}\pBooltype_A$ as the proof:
\begin{center}
\AxiomC{$\proofvdots{\pi_0}$}
\noLine\UnaryInfC{$A^n \vdash A^3$}
\AxiomC{$\cdots$}
\AxiomC{$\proofvdots{\pi_{n-1}}$}
\noLine\UnaryInfC{$A^n \vdash A^3$}
\RightLabel{\scriptsize $\& R$}
\doubleLine\TrinaryInfC{$A^n \vdash (A^3)^n$}
\AxiomC{}
\UnaryInfC{$A \vdash A$}
\RightLabel{\scriptsize $\& L_0$}
\UnaryInfC{$A^3 \vdash A$}
\RightLabel{\scriptsize $\multimap L$}
\BinaryInfC{$A^n, \tNBool_{A^3} \vdash A$}
\RightLabel{\scriptsize $\multimap R$}
\UnaryInfC{$\tNBool_{A^3} \vdash \tNBool_A$}
\DisplayProof
\end{center}
The denotation of ${_n}\pBooltype_A$ is the function
\[
\dntn{{_n}\pBooltype_A}(\varphi)(a_0, ..., a_{n-1}) = \proj_0 \circ \,\varphi((a_0, a_0, a_0), ..., (a_{n-1}, a_{n-1}, a_{n-1})),
\]
and hence $\dntn{{_n}\pBooltype_A}(\dntn{\underline{i}_{A^3}}) = \dntn{\underline{i}_{A}}$.
\end{proof}

We next encode the transition function $\delta: \Sigma \times Q \to \Sigma \times Q \times \{\text{left, right}\}$ of a given Turing machine.

\begin{lemma}\label{encoding functions of n bools}
Given any function $f: \{0, ..., n-1\} \to \{0, ..., m-1\}$, there exists a proof $F$ of $\tNBool_A \vdash \tMBool_A$ which encodes $f$.
\end{lemma} 

\begin{proof}

Let $F$ be the following proof
\begin{center}
\AxiomC{}
\UnaryInfC{$A \vdash A$}
\RightLabel{\scriptsize $\&L_{f(1)}$}
\UnaryInfC{$A^m \vdash A$}
\AxiomC{}
\noLine\UnaryInfC{$...$}
\AxiomC{}
\UnaryInfC{$A \vdash A$}
\RightLabel{\scriptsize $\&L_{f(n)}$}
\UnaryInfC{$A^m \vdash A$}

\RightLabel{\scriptsize $\&R$}
\doubleLine\TrinaryInfC{$A^m \vdash A^n$}

\AxiomC{}
\UnaryInfC{$A \vdash A$}

\RightLabel{\scriptsize $\multimap L$}
\BinaryInfC{$A^m, \tNBool_A \vdash A$}
\RightLabel{\scriptsize $\multimap R$}
\UnaryInfC{$\tNBool_A \vdash \tMBool_A$}

\DisplayProof
\end{center}
The denotation of $F$ is, for $\varphi \in \dntn{\tNBool_A}$ and $(a_0, ..., a_{m-1}) \in \dntn{A^m}$:
\[
\dntn{F}(\varphi)(a_0, ..., a_{m-1}) = \varphi(a_{f(0)}, ..., a_{f(n-1)}).
\]
In particular, this means that $\dntn{F}(\dntn{\underline{i}_A})(a_0, ..., a_{m-1}) = \proj_i(a_{f(0)}, ..., a_{f(n-1)}) = a_{f(i)}$, and hence $\dntn{F}(\dntn{\underline{i}_A}) = \proj_{f(i)} = \dntn{\underline{f(i)}_A}$ as desired.
\end{proof}

\begin{proposition}\label{encoding transition functions}
Given a transition function $\delta: \Sigma \times Q \to \Sigma \times Q \times \{\text{left, right}\}$, write $\delta_i$ for the component $\proj_i \circ \,\delta$ ($i \in \{0,1,2\}$). Then there exists proofs 
\begin{align*}
{^0_\delta}\pTrans_A:&  \; \tBool_A, \tNBool_A \vdash \tBool_A  \\
{^1_\delta}\pTrans_A:& \; \tBool_A, \tNBool_A \vdash \tNBool_A  \\
{^2_\delta}\pTrans_A:& \; \tBool_A, \tNBool_A \vdash \tBool_A  
\end{align*} which encode $\delta_i$, for $i = 0,1,2$. We are using the convention that $\text{left} = 0$ and $\text{right} = 1$.
\end{proposition}

\begin{proof}
Given $i \in \{0,1,2\}$ and and $j \in \Sigma = \{0,1\}$, let $\Delta_{i,j}$ be the proof obtained from Lemma \ref{encoding functions of n bools} corresponding to the function $\delta_i(j,-)$, omitting the final $\multimap R$ rule. Define ${^i_\delta}\pTrans_A$ as the following proof, where $m = n$ if $i = 1$ and $m = 2$ otherwise:
\begin{center}
\AxiomC{$\proofvdots{\Delta_{i,0}}$}
\noLine\UnaryInfC{$A^m, \tNBool_A \vdash A$}
\AxiomC{$\proofvdots{\Delta_{i,1}}$}
\noLine\UnaryInfC{$A^m, \tNBool_A \vdash A$}
\RightLabel{\scriptsize $\& R$}
\BinaryInfC{$A^m, \tNBool_A \vdash A^2$}

\AxiomC{}
\UnaryInfC{$A \vdash A$}
\RightLabel{\scriptsize $\multimap L$}
\BinaryInfC{$A^m, \tBool_A, \tNBool_A \vdash A$}
\RightLabel{\scriptsize $\multimap L$}
\UnaryInfC{$\tBool_A, \tNBool_A \vdash \tMBool_A$}
\DisplayProof
\end{center}
Then $\dntn{{^i_\delta}\pTrans_A}$ is the function
\[
\dntn{{^i_\delta}\pTrans_A}(\psi \otimes \varphi)(a_0, ..., a_{m-1}) = \psi(\varphi(a_{\delta_i(0,0)}, ..., a_{\delta_i(0,n-1)}), \varphi(a_{\delta_i(1,0)}, ..., a_{\delta_i(1,n-1)})),
\]
and thus we have $\dntn{{^i_\delta}\pTrans_A}(\dntn{\underline{\sigma}} \otimes \dntn{\underline{q}}) = \proj_{\delta_i(\sigma, q)} = \dntn{\underline{\delta_i(\sigma, q)}}$.
\end{proof}

Once the new state, symbol and direction have been computed, our remaining task is to recombine the symbols with the binary integers representing the tape.

\begin{example} \label{example: concat and append}
Let $E = A \multimap A$. The proof $\pConcat_A$ is
\begin{center}
\AxiomC{}
\UnaryInfC{$\textcolor{blue}{{!}E} \vdash {!}E$}
\AxiomC{}
\UnaryInfC{$\textcolor{red}{{!}E} \vdash {!}E$}
\AxiomC{}
\UnaryInfC{$\textcolor{blue}{{!}E} \vdash {!}E$}
\AxiomC{}
\UnaryInfC{$\textcolor{red}{{!}E} \vdash {!}E$}
\AxiomC{$\proofvdots{\pComp_A^2}$}
\noLine\UnaryInfC{$A, E, E \vdash A$}
\RightLabel{\scriptsize $\multimap R$}
\UnaryInfC{$E, E \vdash E$}
\RightLabel{\scriptsize $\multimap L$}
\BinaryInfC{$\textcolor{red}{{!}E}, E, \tInt_A \vdash E$}
\RightLabel{\scriptsize $\multimap L$}
\BinaryInfC{$\textcolor{blue}{{!}E}, \textcolor{red}{{!}E}, E, \tBint_A \vdash E$}
\RightLabel{\scriptsize $\multimap L$}
\BinaryInfC{$\textcolor{red}{{!}E}, \textcolor{blue}{{!}E}, \textcolor{red}{{!}E}, \tInt_A, \tBint_A \vdash E$}
\RightLabel{\scriptsize $\multimap L$}
\BinaryInfC{$\textcolor{blue}{{!}E}, \textcolor{red}{{!}E}, \textcolor{blue}{{!}E}, \textcolor{red}{{!}E}, \tBint_A, \tBint_A \vdash E$}
\RightLabel{\scriptsize $2 \times$ ctr}
\doubleLine\UnaryInfC{$\textcolor{blue}{{!}E}, \textcolor{red}{{!}E}, \tBint_A, \tBint_A \vdash E$}
\RightLabel{\scriptsize $2 \times \multimap R$}
\doubleLine\UnaryInfC{$\tBint_A, \tBint_A \vdash \tBint_{A}$}
\DisplayProof
\end{center}
where $\pComp^2_A$ encodes composition \cite[Definition 3.2]{clift_murfet}. The colours indicate which copies of ${!}E$ are contracted together. When cut against the proofs of binary integers $\underline{S}_A$ and $\underline{T}_A$, the resulting proof will be equivalent under cut elimination to $\underline{ST}_A$. We write $\pConcat(S, -)$ for the proof of $\tBint_A \vdash \tBint_A$ obtained by cutting a binary integer $\underline{S}_A$ against $\pConcat_A$ such that the first $\tBint_A$ is consumed; meaning that $\pConcat(S,-)$ prepends by $S$. Similarly define $\pConcat(-,T)$ as the proof which appends by $T$. 
\end{example}

\begin{lemma} \label{appending bools to bints}
Let $W_{00}, W_{01}, W_{10}, W_{11}$ be fixed binary integers, possibly empty. There exists a proof $\pi(W_{00}, W_{01}, W_{10}, W_{11})$ of $\tBint_A, \tBool_A, \tBool_A \vdash \tBint_A$ which encodes \[(S, \sigma, \tau) \mapsto SW_{\sigma\tau}.\]
\end{lemma}
\begin{proof}
Let $E = A \multimap A$. We give a proof corresponding to the simpler function $(S, \sigma) \mapsto SW_\sigma$, where $W_0, W_1$ are fixed binary sequences:
\begin{center}
\AxiomC{$\proofvdots{\pConcat_A(-,W_0)}$}
\noLine\UnaryInfC{$\tBint_A, {!}E, {!}E, A \vdash A$}
\AxiomC{$\proofvdots{\pConcat_A(-,W_1)}$}
\noLine\UnaryInfC{$\tBint_A, {!}E, {!}E, A \vdash A$}
\RightLabel{\scriptsize $\& R$}
\BinaryInfC{$\tBint_A, {!}E, {!}E, A \vdash A\&A$}
\AxiomC{}
\UnaryInfC{$A \vdash A$}
\RightLabel{\scriptsize $\multimap L$}
\BinaryInfC{$\tBint_A, \tBool_A, {!}E, {!}E, A \vdash A$}
\doubleLine\RightLabel{\scriptsize $3\times \multimap R$}
\UnaryInfC{$\tBint_A, \tBool_A \vdash \tBint_A$}
\DisplayProof
\end{center}
The required proof $\pi(W_{00}, W_{01}, W_{10}, W_{11})$ is an easy extension of this, involving two instances of the $\&R$ and $\multimap L$ rules rather than one.
\end{proof}

\begin{proposition} \label{prop:recombination steps}
There exist proofs ${^0}\pRecomb_A$ and ${^1}\pRecomb_A$ of \[\tBint_A, 3\, \tBool_A \vdash \tBint_A\] which encode the functions
\[
(S, \tau, \sigma, d) \mapsto
\begin{cases}
S & \text{if $d = 0$ (left)} \\
S \sigma \tau & \text{if $d = 1$ (right)}
\end{cases}
\quad \text{and} \quad
(T, \tau, \sigma, d) \mapsto
\begin{cases}
T \tau \sigma & \text{if $d = 0$ (left)} \\
T & \text{if $d = 1$ (right)}
\end{cases}
\]
respectively.
\end{proposition}
\begin{proof}
Define $\pi(-,-,-,-)$ as described in Lemma \ref{appending bools to bints}, omitting the final $\multimap R$ rules. The desired proof ${^0}\pRecomb_A$ is:
\begin{center}
\AxiomC{$\proofvdots{\pi(\emptyset, \emptyset, \emptyset, \emptyset)}$}
\noLine\UnaryInfC{$\tBint_A, 2\, \tBool_A, {!}E, {!}E, A  \vdash A$}
\AxiomC{$\proofvdots{\pi(00, 10, 01, 11)}$}
\noLine\UnaryInfC{$\tBint_A, 2\, \tBool_A, {!}E, {!}E, A  \vdash A$}
\RightLabel{\scriptsize $\& R$}
\BinaryInfC{$\tBint_A, 2\, \tBool_A, {!}E, {!}E, A  \vdash A \& A$}
\AxiomC{}
\UnaryInfC{$A \vdash A$}
\RightLabel{\scriptsize $\multimap L$}
\BinaryInfC{$\tBint_A, 3\, \tBool_A, {!}E, {!}E, A \vdash A$}
\doubleLine\RightLabel{\scriptsize $3\times \multimap R$}
\UnaryInfC{$\tBint_A, 3\, \tBool_A \vdash \tBint_A$}
\DisplayProof
\end{center}
and ${^1}\pRecomb_A$ is the same, with the leftmost branch replaced by $\pi(00, 01, 10, 11)$ and the second branch replaced by $\pi(\emptyset, \emptyset, \emptyset, \emptyset)$.
\end{proof}

\begin{proposition} \label{prop: left right and state}
There exist proofs
\begin{align*}
{_\delta}\pLeft_A:&  \; 3\,\tBint_{A^3}, \phantom{1}\,\tBint_{A^3}, 2\,\tNBool_{A^3} \vdash \tBint_A \\
{_\delta}\pRight_A:& \; 2\,\tBint_{A^3}, 2\,\tBint_{A^3}, 2\,\tNBool_{A^3} \vdash \tBint_A \\
{_\delta}\pState_A:& \; \phantom{3}\,\tBint_{A^3},\phantom{1\,\tBint_{A^3}, 1\,}  \tNBool_{A^3} \vdash \tNBool_A 
\end{align*}
which, if fed the indicated number of copies of $S, T$ and $q$ corresponding to a Turing configuration, update the left part of the tape, the right part of the tape, and the state.
\end{proposition}

\begin{proof}
We simply compose (using cuts) the proofs from Propositions \ref{decomposing bints head} through \ref{prop:recombination steps}; the exact sequence of cuts is given in Figures \ref{fig: proof of left} - \ref{fig: proof of state}. The verification that the proofs perform the desired tasks is made clear through the following informal computations. Here $\< S \sigma, T\tau, q\>$ is the configuration of the Turing machine at time $t$, and $\<S', T', q'\>$ is its configuration at time $t+1$. In other words, we have $\delta(\sigma, q) = (\sigma', q', d)$, and \[(S', T') = 
\begin{cases}
(S, T \tau\sigma') & d = 0 \text{ (left)}
\\(S\sigma'\tau, T) & d = 1 \text{ (right).}
\end{cases}\]

\begin{align*}
{_\delta}\pLeft_A \text{ is}:
&\phantom{\xmapsto{\makebox[0.6cm]{}}} (S \sigma)^{\otimes 3} \otimes (T \tau) \otimes q^{\otimes 2}
\\&\xmapsto{\makebox[0.5cm]{}}   S \otimes \sigma^{\otimes 2} \otimes \tau \otimes q^{\otimes 2}
& (\pTail_A \otimes \pHead_A^{\otimes 3} \otimes {_n}\pBooltype_A^{\otimes 2})
\\&\xmapsto{\makebox[0.5cm]{}}   S \otimes \tau \otimes (\sigma \otimes q)^{\otimes 2}
& (\text{exchange})
\\&\xmapsto{\makebox[0.5cm]{}}   S \otimes \tau \otimes \sigma' \otimes d
& (\id^{\otimes 2} \otimes {^0_\delta}\pTrans_A \otimes {^2_\delta}\pTrans_A)
\\&\xmapsto{\makebox[0.5cm]{}}   S'
& ({^0}\pRecomb_A)\end{align*}
\begin{align*}
{_\delta}\pRight_A \text{ is}: 
&\phantom{\xmapsto{\makebox[0.6cm]{}}} (S \sigma)^{\otimes 2} \otimes (T \tau)^{\otimes 2} \otimes q^{\otimes 2}
\\&\xmapsto{\makebox[0.5cm]{}}   \sigma^{\otimes 2} \otimes \tau \otimes T \otimes q^{\otimes 2}
& (\pHead_A^{\otimes 3}  \otimes \pTail_A \otimes {_n}\pBooltype_A^{\otimes 2})
\\&\xmapsto{\makebox[0.5cm]{}}   T \otimes \tau \otimes (\sigma \otimes q)^{\otimes 2}
& (\text{exchange})
\\&\xmapsto{\makebox[0.5cm]{}}   T \otimes \tau \otimes \sigma' \otimes d
& (\id^{\otimes 2} \otimes {^0_\delta}\pTrans_A \otimes {^2_\delta}\pTrans_A)
\\&\xmapsto{\makebox[0.5cm]{}}   T'
& ({^1}\pRecomb_A) \\ \\
{_\delta}\pState_A \text{ is}:
&\phantom{\xmapsto{\makebox[0.6cm]{}}} (S \sigma) \otimes q
\\&\xmapsto{\makebox[0.5cm]{}}   \sigma \otimes q
& (\pHead_A \otimes {_n}\pBooltype_A)
\\&\xmapsto{\makebox[0.5cm]{}}   q'.
& ({^1_\delta}\pTrans_A)
\end{align*}
\end{proof}

\begin{theorem} \label{thm: single transition step}
There exists a proof ${_\delta}\pStep_A$ of $\tTur_{A^3} \vdash \tTur_A$ which encodes a single transition step of a given Turing machine. 
\end{theorem}

\begin{proof}
The desired proof ${_\delta}\pStep_A$ is given in Figure \ref{fig: proof of single step transition}.
\end{proof}

By cutting the above construction against itself, we obtain:
\begin{corollary}\label{cor:step_p}
For each $p \geq 1$, there exists a proof ${^p_\delta}\pStep_A$ of $\tTur_{A^{3^p}} \vdash \tTur_A$ which encodes $p$ transition steps of a given Turing machine.
\end{corollary}

Note that this iteration must be performed `by hand' for each $p$; we cannot iterate for a variable number of steps. By this we mean that it is not possible to devise a proof of $\tInt_{B}, \tTur_{C} \vdash \tTur_A$ (for suitable types $B,C$) which simulates a given Turing machine for $n$ steps when cut against the Church numeral $\underline{n}_B$. The fundamental problem is that iteration using $\tInt_{B}$ only allows iteration of \emph{endomorphisms} $B \multimap B$, and so the fact that our base type changes in each iteration of ${_\delta}\pStep_A$ makes this impossible.

\begin{remark}\label{remark:second-order-final}
If one is willing to use second-order, then iteration in the above sense becomes possible via the following proof, where $\tTur = \forall A. \tTur_A$:
\begin{center}
\AxiomC{$\proofvdots{{_\delta}\pStep_A}$}
\noLine\UnaryInfC{$\tTur_{A^3} \vdash \tTur_A$}
\RightLabel{\scriptsize $\forall L$}
\UnaryInfC{$\tTur \vdash \tTur_A$}
\RightLabel{\scriptsize $\forall R$}
\UnaryInfC{$\tTur \vdash \tTur$}
\RightLabel{\scriptsize $\multimap R$}
\UnaryInfC{$\vdash \tTur \multimap \tTur$}
\RightLabel{\scriptsize prom}
\UnaryInfC{$\vdash {!}(\tTur \multimap \tTur)$}

\AxiomC{}
\UnaryInfC{$\tTur \vdash \tTur$}
\AxiomC{}
\UnaryInfC{$\tTur \vdash \tTur$}
\RightLabel{\scriptsize $\multimap L$}
\BinaryInfC{$\tTur \multimap \tTur, \tTur \vdash \tTur$}
\RightLabel{\scriptsize prom}
\BinaryInfC{$\tInt_{\tTur}, \tTur \vdash \tTur$}
\RightLabel{\scriptsize $\forall L$}
\UnaryInfC{$\tInt, \tTur \vdash \tTur$}
\DisplayProof
\end{center}
In the original encoding given by Girard \cite{girard_complexity} the use of second-order quantifiers is present throughout the encoding. In contrast, the above encoding only involves the $\forall R, \forall L$ rules at the very bottom of the proof tree, which makes it more suitable for the study of Turing machines in the context of differential linear logic in \cite{clift_murfet3}.
\end{remark}

\begin{figure}
\centering
\begin{tikzpicture}[every node/.style={block},
        block/.style={minimum height=1.5em,outer sep=0pt,draw,rectangle,node distance=0pt}]
    \node at (1,0)(A) 
{
\AxiomC{$\proofvdots{\pTail_A \otimes \pHead_A^{\otimes 3} \otimes {_n}\pBooltype_A^{\otimes 2}}$}
\noLine\UnaryInfC{$3\, \tBint_{A^3}, \tBint_{A^3}, 2\,\tNBool_{A^3} \vdash \tBint_A \otimes \tBool_A^{\otimes 3} \otimes \tNBool_A^{\otimes 2}$}
\DisplayProof 
};
    \node at (3,-2.8)(B)
{
\AxiomC{$\proofvdots{\id^{\otimes 2} \otimes {^0_\delta}\pTrans_A \otimes {^2_\delta}\pTrans_A}$}
\noLine\UnaryInfC{$\tBint_A, \tBool_A, 2\,(\tBool_A, \tNBool_A) \vdash \tBint_A \otimes \tBool_A^{\otimes 3}$}
\RightLabel{\scriptsize exch}
\UnaryInfC{$\tBint_A, 3\,\tBool_A, 2\,\tNBool_A \vdash \tBint_A \otimes \tBool_A^{\otimes 3}$}
\RightLabel{\scriptsize $\otimes L$}
\doubleLine\UnaryInfC{$\tBint_A \otimes \tBool_A^{\otimes 3} \otimes \tNBool_A^{\otimes 2} \vdash \tBint_A \otimes \tBool_A^{\otimes 3}$}
\DisplayProof
};
    \node at (4,-5.8)(C)
{
\AxiomC{$\proofvdots{{^0}\pRecomb_A}$}
\noLine\UnaryInfC{$\tBint_A, 3\, \tBool_A \vdash \tBint_A$}
\RightLabel{\scriptsize $\otimes L$}
\doubleLine\UnaryInfC{$\tBint_A \otimes \tBool_A^{\otimes 3} \vdash \tBint_A$}
\DisplayProof
};
    \node at (0, -7.8)(D) [draw=none]
{
\AxiomC{}
\RightLabel{\scriptsize cut}
\UnaryInfC{$3\, \tBint_{A^3}, \tBint_{A^3}, 2\, \tNBool_{A^3} \vdash \tBint_A \otimes \tBool_A^{\otimes 3}$}
\AxiomC{\qquad}
\RightLabel{\scriptsize cut}
\BinaryInfC{$3\, \tBint_{A^3}, \tBint_{A^3}, 2\, \tNBool_{A^3} \vdash \tBint_A$}
\DisplayProof 
};
    \draw[-latex] (-3,-0.98) -- (-3,-7.28); 
    \draw[-latex] (-0.2,-4.4) -- (-0.2,-7.28); 
    \draw[-latex] (4.2,-7.035) -- (4.2,-7.92); 
\end{tikzpicture}
\caption{The proof ${_\delta}\pLeft_A$.} \label{fig: proof of left}
\end{figure}

\begin{figure}
\centering
\begin{tikzpicture}[every node/.style={block},
        block/.style={minimum height=1.5em,outer sep=0pt,draw,rectangle,node distance=0pt}]
    \node at (1,0)(A) 
{
\AxiomC{$\proofvdots{\pHead_A^{\otimes 3}  \otimes \pTail_A \otimes {_n}\pBooltype_A^{\otimes 2}}$}
\noLine\UnaryInfC{$2\, \tBint_{A^3}, 2\,\tBint_{A^3}, 2\,\tNBool_{A^3} \vdash \tBool_A^{\otimes 3} \otimes \tBint_A \otimes \tNBool_A^{\otimes 2}$}
\DisplayProof 
};
    \node at (3,-2.8)(B)
{
\AxiomC{$\proofvdots{\id^{\otimes 2} \otimes {^0_\delta}\pTrans_A \otimes {^2_\delta}\pTrans_A}$}
\noLine\UnaryInfC{$\tBint_A, \tBool_A, 2\,(\tBool_A, \tNBool_A) \vdash \tBint_A \otimes \tBool_A^{\otimes 3}$}
\RightLabel{\scriptsize exch}
\UnaryInfC{$3\, \tBool_A, \tBint_A, 2\, \tNBool_A \vdash \tBint_A \otimes \tBool_A^{\otimes 3}$}
\RightLabel{\scriptsize $\otimes L$}
\doubleLine\UnaryInfC{$\tBool_A^{\otimes 3} \otimes \tBint_A \otimes \tNBool_A^{\otimes 2} \vdash \tBint_A \otimes \tBool_A^{\otimes 3}$}
\DisplayProof
};
    \node at (4,-5.8)(C)
{
\AxiomC{$\proofvdots{{^1}\pRecomb_A}$}
\noLine\UnaryInfC{$\tBint_A, 3\, \tBool_A \vdash \tBint_A$}
\RightLabel{\scriptsize $\otimes L$}
\doubleLine\UnaryInfC{$\tBint_A \otimes \tBool_A^{\otimes 3} \vdash \tBint_A$}
\DisplayProof
};
    \node at (0, -7.8)(D) [draw=none]
{
\AxiomC{}
\RightLabel{\scriptsize cut}
\UnaryInfC{$2\, \tBint_{A^3}, 2\,\tBint_{A^3}, 2\, \tNBool_{A^3} \vdash \tBint_A \otimes \tBool_A^{\otimes 3}$}
\AxiomC{\qquad}
\RightLabel{\scriptsize cut}
\BinaryInfC{$2\, \tBint_{A^3}, 2\,\tBint_{A^3}, 2\, \tNBool_{A^3} \vdash \tBint_A$}
\DisplayProof 
};
    \draw[-latex] (-3,-0.98) -- (-3,-7.28); 
    \draw[-latex] (-0.2,-4.4) -- (-0.2,-7.28); 
    \draw[-latex] (4.2,-7.035) -- (4.2,-7.92); 
\end{tikzpicture}

\caption{The proof ${_\delta}\pRight_A$.} \label{fig: proof of right}
\end{figure}

\begin{figure}
\centering
\AxiomC{$\proofvdots{\pHead_A \otimes {_n}\pBooltype_A}$}
\noLine\UnaryInfC{$\tBint_{A^3}, \tNBool_{A^3} \vdash \tBool_A \otimes \tNBool_A$}
\AxiomC{$\proofvdots{{^1_\delta}\pTrans_A}$}
\noLine\UnaryInfC{$\tBool_A, \tNBool_A \vdash \tNBool_A$}
\RightLabel{\scriptsize $\otimes L$}
\UnaryInfC{$\tBool_A \otimes \tNBool_A \vdash \tNBool_A$}
\RightLabel{\scriptsize cut}
\BinaryInfC{$\tBint_{A^3}, \tNBool_{A^3} \vdash \tNBool_A$}
\DisplayProof 

\caption{The proof ${_\delta}\pState_A$.} \label{fig: proof of state}
\end{figure}

\begin{figure}
\centering
\begin{tikzpicture}[every node/.style={block},
        block/.style={minimum height=1.5em,outer sep=0pt,draw,rectangle,node distance=0pt}]
    \node at (1,0)(A) 
{
\AxiomC{$\proofvdots{{_\delta}\pLeft_A}$}
\noLine\UnaryInfC{$3\, \tBint_{A^3}, \tBint_{A^3}, 2\,\tNBool_{A^3} \vdash \tBint_A$}
\RightLabel{\scriptsize der, prom}
\doubleLine\UnaryInfC{$3\, {!}\tBint_{A^3}, {!}\tBint_{A^3}, 2\, {!}\tNBool_{A^3} \vdash {!}\tBint_A$}
\DisplayProof 
};
    \node at (3,-2.6)(B)
{
\AxiomC{$\proofvdots{{_\delta}\pRight_A}$}
\noLine\UnaryInfC{$2\, \tBint_{A^3}, 2\,\tBint_{A^3}, 2\,\tNBool_{A^3} \vdash \tBint_A$}
\RightLabel{\scriptsize der, prom}
\doubleLine\UnaryInfC{$2\, {!}\tBint_{A^3}, 2\,{!}\tBint_{A^3}, 2\, {!}\tNBool_{A^3} \vdash {!}\tBint_A$}
\DisplayProof
};
    \node at (4,-5.2)(C)
{
\AxiomC{$\proofvdots{{_\delta}\pState_A}$}
\noLine\UnaryInfC{$\tBint_{A^3}, \tNBool_{A^3} \vdash \tNBool_A$}
\RightLabel{\scriptsize der, prom}
\doubleLine\UnaryInfC{$ {!}\tBint_{A^3}, {!}\tNBool_{A^3} \vdash {!}\tNBool_A$}
\DisplayProof
};
    \node at (0, -8)(D) [draw=none]
{
\AxiomC{}
\RightLabel{\scriptsize $\otimes R$}
\UnaryInfC{$5\, {!}\tBint_{A^3}, 3\,{!}\tBint_{A^3}, 4\, {!}\tNBool_{A^3} \vdash {!}\tBint_A\otimes {!}\tBint_A$}
\RightLabel{\scriptsize $\otimes R$}
\AxiomC{\qquad}
\BinaryInfC{$6\, {!}\tBint_{A^3}, 3\,{!}\tBint_{A^3}, 5\, {!}\tNBool_{A^3} \vdash \tTur_A$}
\RightLabel{\scriptsize ctr}
\doubleLine\UnaryInfC{${!}\tBint_{A^3}, {!}\tBint_{A^3}, {!}\tNBool_{A^3} \vdash \tTur_A$}
\RightLabel{\scriptsize $\otimes L$}
\UnaryInfC{$\tTur_{A^3} \vdash \tTur_A$}
\DisplayProof 
};
    \draw[-latex] (-2.5,-1.15) -- (-2.5,-6.98); 
    \draw[-latex] (-0.5,-3.79) -- (-0.5,-6.98); 
    \draw[-latex] (4.5,-6.34) -- (4.5,-7.51); 
\end{tikzpicture}
\caption{The proof ${_\delta}\pStep_A$.} \label{fig: proof of single step transition}
\end{figure}

\newpage

Next we observe that ${_\delta} \pStep_A$ is component-wise plain in the sense of Definition \ref{defn:plain_comp}.

\begin{definition} We write ${_\delta}\pLeft_A^\dagger$ for the proof
\begin{center}
\AxiomC{$\proofvdots{{_\delta}\pLeft_A}$}
\noLine\UnaryInfC{$3\,\tBint_{A^3}, \tBint_{A^3}, 2\,\tNBool_{A^3} \vdash \tBint_A$}
\RightLabel{\scriptsize der}
\doubleLine\UnaryInfC{$3\,{!}\tBint_{A^3}, {!}\tBint_{A^3}, 2\,{!}\tNBool_{A^3} \vdash \tBint_A$}
\RightLabel{\scriptsize ctr}
\doubleLine\UnaryInfC{${!}\tBint_{A^3}, {!}\tBint_{A^3}, {!}\tNBool_{A^3} \vdash \tBint_A$}
\DisplayProof
\end{center}
and similarly for ${_\delta}\pRight_A^\dagger$ and ${_\delta}\pState_A^\dagger$.
\end{definition}

It is clear that:

\begin{lemma}\label{lemma:stepsareplain} The proofs ${_\delta}\pLeft_A^\dagger, {_\delta}\pRight_A^\dagger, {_\delta}\pState_A^\dagger$ are plain.
\end{lemma}

\begin{proposition}\label{prop:pstepplain} The proofs ${_\delta} \pStep_A$ and ${^p_\delta} \pStep_A$ are component-wise plain.
\end{proposition}
\begin{proof}
The second claim follows from Proposition \ref{prop:cut_componentwiseplain}.
\end{proof}

The denotation of the proof ${_\delta} \pStep_A$ is a linear map
\[
\dntn{{_\delta}\pStep_A}: {!} \dntn{\tBint_{A^3}}^{\otimes 2} \otimes {!} \dntn{\tBool_{A^3}} \lto {!} \dntn{\tBint_{A}}^{\otimes 2} \otimes {!} \dntn{\tBool_{A}}
\]
and we compute the value of this map on vectors of the form
\[\Psi = \vacu_{\alpha} \otimes \vacu_{\beta} \otimes \vacu_{\gamma},\] 
where
\[
\alpha = \Sum_{i=1}^s a_i \dntn{\underline{S_i \sigma_i}}, \quad \beta = \Sum_{i=1}^t b_i \dntn{\underline{T_i \tau_i}}, \quad \gamma = \Sum_{i=1}^r c_i \dntn{\underline{q_i}}.
\]
Note that
\[
\dntn{{_\delta}\pStep_A}(\Psi) = \vacu_{\dntn{{_\delta}\pLeft_A}(\alpha^{\otimes 3} \otimes \beta \otimes \gamma^{\otimes 2})} \otimes \vacu_{\dntn{{_\delta}\pRight_A}(\alpha^{\otimes 2} \otimes \beta^{\otimes 2} \otimes \gamma^{\otimes 2})} \otimes \vacu_{\dntn{{_\delta}\pState_A}(\alpha \otimes \gamma)}, 
\]
and so the problem reduces to computing the values of each of $\dntn{{_\delta}\pLeft}$,  $\dntn{{_\delta}\pRight}$, and $\dntn{{_\delta}\pState}$ on the appropriate number of copies of $\alpha, \beta, \gamma$. Write $(\hat{\sigma}_i^j, \hat{q}_i^j, \hat{d}_i^j) = \delta(\sigma_i, q_j)$. We have:
\begin{itemize}
    \item $\dntn{{_\delta}\pLeft_A}(\dntn{\underline{S_i \sigma_i}} \otimes \dntn{\underline{S_j \sigma_j}} \otimes \dntn{\underline{S_k \sigma_k}} \otimes \dntn{\underline{T_l \tau_l}} \otimes \dntn{\underline{q_m}}\otimes\dntn{\underline{q_n}}) 
    = \delta_{\hat{d}_k^n = 0} \dntn{\underline{S_i}} + \delta_{\hat{d}_k^n = 1} \dntn{\underline{S_i \hat{\sigma}_j^m \tau_l}}$,
    \item $\dntn{{_\delta}\pRight_A}(\dntn{\underline{S_j \sigma_j}} \otimes \dntn{\underline{S_k \sigma_k}} \otimes \dntn{\underline{T_l \tau_l}} \otimes \dntn{\underline{T_i \tau_i}} \otimes \dntn{\underline{q_m}}\otimes\dntn{\underline{q_n}}) 
    = \delta_{\hat{d}_k^n = 0}\dntn{\underline{T_i \tau_l \hat{\sigma}_j^m}}  + \delta_{\hat{d}_k^n = 1} \dntn{\underline{T_i}}$,
    \item $\dntn{{_\delta}\pState_A}(\dntn{\underline{S_i \sigma_i}} \otimes \dntn{\underline{q_j}}) 
    = \dntn{\underline{\hat{q}_i^j}}$,
\end{itemize}
where $\delta$ on the right is the Kronecker delta. Using this, we compute
\allowdisplaybreaks

\begin{lemma}\label{lemma:denotation_step} The vector $\dntn{{_\delta}\pLeft_A}(\alpha^{\otimes 3} \otimes \beta \otimes \gamma^{\otimes 2})$ is equal as an element of $\dntn{\tBint_A}$ to
\begin{align}
\;&
\left(\Sum_{i=1}^s \Sum_{j=1}^r a_ic_j \delta_{\hat{d}_i^j = 0}\right)
\left(\Sum_{i=1}^s a_i\right)
\left(\Sum_{i=1}^t b_i\right)
\left(\Sum_{i=1}^r c_i\right)
\left(\Sum_{i=1}^s a_i\dntn{\underline{S_i}} \right)\notag
\\& 
+ 
\left(\Sum_{i=1}^s \Sum_{j=1}^r a_ic_j \delta_{\hat{d}_i^j = 1}\right)
\left(\Sum_{i=1}^s\Sum_{j=1}^s \Sum_{k=1}^t \Sum_{l=1}^r a_i a_j b_k c_l \dntn{\underline{S_i \hat{\sigma}_j^l \tau_k}}\right)\,. \label{eq:leftdenotation}
\end{align}
The vector $\dntn{{_\delta}\pRight_A}(\alpha^{\otimes 2} \otimes \beta^{\otimes 2} \otimes \gamma^{\otimes 2})$ is equal as an element of $\dntn{\tBint_A}$ to
\begin{align}
\;&
\left(\Sum_{i=1}^s \Sum_{j=1}^r a_ic_j \delta_{\hat{d}_i^j = 1}\right)
\left(\Sum_{i=1}^s a_i\right)
\left(\Sum_{i=1}^t b_i\right)
\left(\Sum_{i=1}^r c_i\right)
\left(\Sum_{i=1}^t b_i\dntn{\underline{T_i}} \right)\notag
\\& 
+ 
\left(\Sum_{i=1}^s \Sum_{j=1}^r a_ic_j \delta_{\hat{d}_i^j = 0}\right)
\left(\Sum_{i=1}^t\Sum_{j=1}^s \Sum_{k=1}^t \Sum_{l=1}^r b_i a_j b_k c_l \dntn{\underline{T_i \hat{\sigma}_j^l \tau_k}}\right)\,. \label{eq:rightdenotation}
\end{align}
The vector $\dntn{{_\delta}\pState_A}(\alpha \otimes \gamma)$ is equal as an element of $\dntn{\tBool_A}$ to
\be\label{eq:statedenotation}
\Sum_{i=1}^s \Sum_{j=1}^r a_ic_j \dntn{\underline{\hat{q}_i^j}}.
\ee
\end{lemma}

\begin{remark} These formulas only describe the values of $\dntn{{_\delta} \pStep_A}$ on certain group-like elements, but the values on more general kets can be derived from these formulas by taking derivatives; see \cite[Corollary 4.5]{clift_murfet3}. 
\end{remark}

\begin{remark} From the calculations in the lemma we can extract polynomial functions, as explained in Section \ref{section:denotations_plain}. To do this we restrict the domain to sequences of bounded length, and identify such sequences with proofs. By Remark \ref{remark:findgoodA} as long as we choose our type $A$ such that $\dim\dntn{A} > c/2$ the set of denotations
\[
\big\{ \dntn{\underline{S}_A} \big\}_{S \in \Sigma^{\le c}}
\]
is a linearly independent set in $\dntn{\tBint_A}$, and in particular these denotations are all distinct, and we may identify $\Sigma^{\le c}$ with the set of proofs $\cat{P}_c = \{ \underline{S}_A \}_{S \in \Sigma^{\le c}}$. Then by Proposition \ref{prop:fpi} if we fix integers $a,b$ and then choose $A$ such that $\dim(\dntn{A}) > \tfrac{1}{2}\max\{a+1,b+1\}$ there is a unique function $F_{{_\delta} \pLeft_A}$ making the diagram
\be
\xymatrix@C+3pc@R+1pc{
{!} \dntn{\tBint_{A^3}} \otimes {!} \dntn{\tBint_{A^3}} \otimes {!} \dntn{\tBool_{A^3}} \ar[r]^-{\dntn{{_\delta} \pLeft_A}} & \dntn{\tBint_A}\\
k\Sigma^{\le a} \times k\Sigma^{\le b} \times kQ \ar[u]^-{\iota}\ar[r]_-{F_{{_\delta} \pLeft_A}} & k \Sigma^{\le a+1} \ar[u]_-{\dntn{-}}
}
\ee
commute. The formula for the function $F_{{_\delta} \pLeft_A}$ is given by reading the formula \eqref{eq:leftdenotation} with $\dntn{\underline{S_i}}$ replaced by $S_i$ and $\dntn{\underline{S_i \hat{\sigma}_j^l \tau_k}}$ replaced by $S_i \hat{\sigma}_j^l \tau_k$, and similarly for $F_{{_\delta} \pRight_A}$ and $F_{{_\delta} \pState_A}$.
\end{remark}

\subsection{The Boolean version}\label{section:actionofstep}

Our ultimate purpose in giving the encodings in this paper is to understand the derivatives of Turing machines in \cite{clift_murfet3}, and we are therefore interested in encodings which allow us to differentiate the encoding with respect to the contents of individual tape squares. One natural way to obtain such an encoding is to derive it from $\pStep$ and that is our goal in this section. In Section \ref{section:direct_boolean} we give encodings obtained directly, without going via $\pStep$.

The aim in this section is to construct a proof ${^p}\pBoolstep$ of
\be\label{eq:typecstep}
a \,{!} \tBool_B, b\, {!} \tBool_B, {!} {_n} \tBool_B \vdash \big({!} \tBool_A\big)^{\otimes c } \otimes \big({!} \tBool_A\big)^{\otimes d} \otimes {!} {_n} \tBool_A
\ee
for some power $B = A^{g(c,d,p)}$ such that if initially the state is $q$ and the tape reads
\[
\ldots,\, \Box,\, x_1,\,\ldots\,,\, \underline{x_a},\, y_1,\, \ldots\,,\, y_b,\, \Box,\, \ldots
\]
with the underline indicating the position of the head, then running the Turing machine for $p$ steps on this input yields
\[
\ldots\,, \Box,\, x'_1,\,\ldots\,, \underline{x'_{c}},\, y'_1,\, \ldots\,, y'_{d},\, \Box,\, \ldots
\]
The algorithm ${^p}\pBoolstep$ returns this state, as a sequence of booleans, in the given order. In this section to avoid an explosion of notation we drop the subscript $\delta$ and write $\pBoolstep$ when we should more correctly write ${_\delta} \pBoolstep$.

\begin{lemma} There is a proof $\pCast_A$ of $\tBool_A \vdash \tBint_A$ which converts a boolean to the equivalent binary sequence; that is, it encodes $\dntn{\underline{i}_A} \mapsto \dntn{\underline{i}_A}$ for $i \in \{0,1\}$.
\end{lemma}
\begin{proof}
Let $E = A \multimap A$. The proof is:
\begin{center}
\AxiomC{$\proofvdots{\underline{0}_A}$}
\noLine\UnaryInfC{${!}E, {!}E, A \vdash A$}
\AxiomC{$\proofvdots{\underline{1}_A}$}
\noLine\UnaryInfC{${!}E, {!}E, A \vdash A$}
\RightLabel{\scriptsize $\& R$}
\BinaryInfC{${!}E, {!}E, A \vdash A\&A$}
\AxiomC{}
\UnaryInfC{$A \vdash A$}
\RightLabel{\scriptsize $\multimap L$}
\BinaryInfC{$\tBool_A, {!}E, {!}E, A \vdash A$}
\doubleLine\RightLabel{\scriptsize $3\times \multimap R$}
\UnaryInfC{$\tBool_A \vdash \tBint_A$}
\DisplayProof
\end{center}
where the incoming proofs at the top are the binary integers $0,1$.
\end{proof}

\begin{lemma} There is a proof $\pRead^j_A$ of $\tBint_{A^{3^{j+2}}} \vdash \tBool_A$ for $j \ge 0$ which ``reads'' the symbol at the position $j$ from the right in a binary integer; that is, it encodes
\[
\dntn{\underline{s_r s_{r-1} \cdots s_1s_0}_{A^{3^{j+2}}}} \longmapsto \dntn{\underline{s_j}_A}\,.
\]
\end{lemma}
\begin{proof}
We use the proofs $\pHead_A : \tBint_{A^3} \vdash \tBint_A$ of Lemma \ref{decomposing bints head} and $\pTail_A: \tBint_{A^3} \vdash \tBint_A$ of Lemma \ref{decomposing bints tail}, cut together as follows (where $\psi \l \phi$ denotes the cut of $\psi, \phi$)
\[
\pHead_{A^{3^{j+1}}} \,\l\, \pTail_{A^{3^j}}\,\l\, \cdots \,\l\,\pTail_{A^3} \,\l\,\pTail_A : \tBint_{A^{3^{j+2}}} \vdash \tBool_A
\]
where there are $j$ copies of $\pTail$ at various base types. Note that if we apply $\pRead^j$ to the empty binary sequence, or to any sequence of length $\le j$, we obtain the boolean $\underline{0}$.
\end{proof}

\begin{lemma} There is a proof $\pMultread^a_A$ of ${!} \tBint_{A^{3^{a+1}}} \vdash \big({!}\tBool_A\big)^{\otimes a}$ which reads off $a$ symbols from the right end of a binary integer; that is, it encodes
\[
\dntn{\underline{s_r s_{r-1} \cdots s_1s_0}_{A^{3^{j+2}}}} \longmapsto \big( \dntn{\underline{s_{a-1}}_A}, \ldots, \dntn{\underline{s_{0}}_A} \big)\,.
\]
\end{lemma}
\begin{proof}
We use the proofs
\begin{align*}
\pRead^{a-1}_A &: \tBint_{A^{3^{a+1}}} \vdash \tBool_A\\
\pRead^{a-2}_{A^3} &: \tBint_{A^{3^{a+1}}} \vdash \tBool_{A^3}\\
\pRead^{a-3}_{A^{3^2}} &: \tBint_{A^{3^{a+1}}} \vdash \tBool_{A^3}\\
&\vdots\\
\pRead^0_{A^{3^{a-1}}} &: \tBint_{A^{3^{a+1}}} \vdash \tBool_{A^{3^{a-1}}}
\end{align*}
to read off the relevant symbols, and then fix the base types by cutting against an appropriate number of copies of the proof ${_n}\pBooltype_B$ of Lemma \ref{fixing_bool_types} with various base types $B$. This leaves us with $a$ proofs of $\tBint_{A^{3^{a+1}}} \vdash \tBool_A$ which we then derelict on the left and promote on the right, and then tensor together to obtain a proof of
\[
a\, {!} \tBint_{A^{3^{a+1}}} \vdash \big( {!} \tBool_A )^{\otimes a}\,.
\]
A series of contractions then gives the desired proof $\pMultread^a_A$.
\end{proof}

\begin{definition} For $c,d \ge 1$ the proof $\pUnpack^{c,d}_A$ of
\[
\tTur_{A^{3^{\max\{c,d\}+1}}} \vdash \big({!} \tBool_A\big)^{\otimes c } \otimes \big({!} \tBool_A\big)^{\otimes d} \otimes {!} {_n} \tBool_A
\]
is defined as follows: let $e = \max\{c,d\}$ and consider the proofs
\begin{align*}
\pMultread^c_{A^{3^{e-c}}} &: {!} \tBint_{A^{3^{e+1}}} \vdash \big({!}\tBool_{A^{3^{e-c}}}\big)^{\otimes c}\\
\pMultread^d_{A^{3^{e-d}}} &: {!} \tBint_{A^{3^{e+1}}} \vdash \big({!}\tBool_{A^{3^{e-d}}}\big)^{\otimes d}\,.
\end{align*}
We cut these proofs against an appropriate number of promoted copies of ${_n}\pBooltype_B$ for various base types $B$, to obtain proofs of sequents with conclusion $\big({!}\tBool_{A}\big)^{\otimes c}, \big({!}\tBool_{A}\big)^{\otimes d}$. These are tensored with the identity on ${!} {}_n \tBool_A$ and with the proof that reverses the output of the $\pMultread$ corresponding to the right hand part of the tape.
\end{definition}

\begin{definition} The proof $\pPack^{a,b}_A$ of $a \,{!} \tBool_A, b\, {!} \tBool_A, {!} {_n} \tBool_A \vdash \tTur_A$ is given by the tensor product of the identity on ${!}{_n} \tBool_A$ and the proofs
\begin{center}
\AxiomC{$\proofvdots{\pCast_A^{\otimes a}}$}
\noLine\UnaryInfC{$a\,\tBool_A \vdash \tBint_A^{\otimes a}$}
\AxiomC{$\proofvdots{\pConcat}$}
\noLine\UnaryInfC{$\tBint^{\otimes a}_A \vdash \tBint_A$}
\RightLabel{\scriptsize cut}
\BinaryInfC{$a\,\tBool_A \vdash \tBint_A$}
\RightLabel{\scriptsize der,prom}
\doubleLine\UnaryInfC{$a\,{!}\tBool_A \vdash {!}\tBint_A$}
\DisplayProof
\end{center}
and
\begin{center}
\AxiomC{$\proofvdots{\pCast_A^{\otimes b}}$}
\noLine\UnaryInfC{$b\,\tBool_A \vdash \tBint_A^{\otimes b}$}
\RightLabel{\scriptsize ex}
\doubleLine\UnaryInfC{$b\,\tBool_A \vdash \tBint_A^{\otimes b}$}
\AxiomC{$\proofvdots{\pConcat}$}
\noLine\UnaryInfC{$\tBint^{\otimes b}_A \vdash \tBint_A$}
\RightLabel{\scriptsize cut}
\BinaryInfC{$b\,\tBool_A \vdash \tBint_A$}
\RightLabel{\scriptsize der,prom}
\doubleLine\UnaryInfC{$b\,{!}\tBool_A \vdash {!}\tBint_A$}
\DisplayProof
\end{center}
where the exchange rule reverses the order of the inputs.
\end{definition}
 
To simplify the formulas in the next definition, let us write
\begin{align*}
\Gamma^{a,b}_A &= a \,{!} \tBool_A, b\, {!} \tBool_A, {!} {_n} \tBool_A\,,\\
X^{c,d}_A &= \big({!} \tBool_A\big)^{\otimes c } \otimes \big({!} \tBool_A\big)^{\otimes d} \otimes {!} {_n} \tBool_A\,.
\end{align*} 

\begin{definition}\label{defn:cstep} For $a,b,c,d,p \ge 1$ the proof ${^p}\pBoolstep_A^{a,b,c,d}$ is
\begin{center}
\AxiomC{$\proofvdots{\pPack^{a,b}_{A^{3^{e+1}}}}$}
\noLine\UnaryInfC{$\Gamma^{a,b}_{A^{3^{p+e+1}}} \vdash \tTur_{A^{3^{p+e+1}}}$}
\AxiomC{$\proofvdots{{^p_\delta}\pStep_{A^{3^{e+1}}}}$}
\noLine\UnaryInfC{$\tTur_{A^{3^{p+e+1}}} \vdash \tTur_{A^{3^{e+1}}}$}
\RightLabel{\scriptsize cut}
\BinaryInfC{$\Gamma^{a,b}_{A^{3^{p+e+1}}} \vdash \tTur_{A^{3^{e+1}}}$}
\AxiomC{$\proofvdots{\pUnpack^{c,d}_A}$}
\noLine\UnaryInfC{$\tTur_{A^{3^{e+1}}} \vdash X^{c,d}_A$}
\RightLabel{\scriptsize cut}
\BinaryInfC{$\Gamma^{a,b}_{A^{3^{p+e+1}}} \vdash X^{c,d}_A$}
\DisplayProof
\end{center}
where $e = \max\{c,d\}$.
\end{definition}

\begin{lemma}\label{lemma:stepbit_plain} The proof ${^p}\pBoolstep^{a,b,c,d}_A$ is component-wise plain.
\end{lemma}
\begin{proof}
By Proposition \ref{prop:cut_componentwiseplain} and Proposition \ref{prop:pstepplain} it suffices to argue that $\pPack$ and $\pUnpack$ are component-wise plain, but this is obvious.
\end{proof}

\section{Direct Boolean encodings}\label{section:direct_boolean}

In the previous section, we developed an encoding of the step function as an operation on a sequence of booleans, rather than a pair of binary integers. It is natural to ask if such an encoding can be designed directly, without going via $\pStep$. 

Ideally, one could define a family of proofs of
\[
a\, {!}\tBool_A, b\, {!}\tBool_A, {!}\tNBool_A \vdash {!}\tBool_A^{\otimes c}\otimes{!}\tBool_A^{\otimes d} \otimes {!}\tNBool_A
\]
indexed by $a,b,c,d$, as in Section \ref{section:actionofstep}. However there is an issue with this approach, in that it cannot be easily iterated. More specifically, we cannot simply cut multiple one-step proofs against one another, since the computation of all steps beyond the first must know where the tape head has moved to, which is invisible when encoding the tape contents as simply a string of booleans.

In this section, we present two possible remedies for this issue. The first is to increase the number of booleans representing the tape after each step, in such a way to ensure that the head is always in the centre of the working section of the tape. We call this the \textit{relative encoding} and denote the proof $\pRelStep$ (see Section \ref{section:relstepdefn}). The second solution is to simply keep track of the head position through the use of an additional $h$-boolean. We call this the \textit{absolute encoding} and denote the proof $\pAbsStep$ (see Section \ref{section:absolutestep}). While this has the advantage that the positions of each tape square do not change after each step, it is limited by the fact that one must remain in a section of the tape of length $h$.

Fix a finite set of states $Q = \{0, ..., n-1\}$ and a finite tape alphabet $\Sigma = \{0, ..., s-1\}$, and let $\delta: \Sigma \times Q \to \Sigma \times Q \times \{\text{left, right}\}$ be the transition function. For $i \in \{0,1,2\}$, we write $\delta_i = \proj_i \circ\, \delta$ for the $i$th component of $\delta$. Throughout $A$ is a fixed type, and we write  $\tBool$ for $\tBool_A$. We give the versions of $\pRelStep$ and $\pAbsStep$ in which the head of the Turing machine must either move left or right in each time step; the modification to allow the head to remain stationary is routine.

\begin{definition}
Given $n \in \NN$, we write ${^n}\pEval$ for the following proof, whose denotation is the evaluation map.
\begin{center}
\AxiomC{}
\UnaryInfC{$A^n \vdash A^n$}
\AxiomC{}
\UnaryInfC{$A \vdash A$}
\RightLabel{\scriptsize $\multimap L$}
\BinaryInfC{$\tNBool, A^n \vdash A$}
\DisplayProof
\end{center}
\end{definition}

\begin{definition}
Given two proofs $\pi, \rho$ of $\Gamma \vdash A$, we write $\pi \,\&\, \rho$ for the following proof.
\begin{center}
\AxiomC{$\proofvdots{\pi}$}
\noLine\UnaryInfC{$\Gamma \vdash A$}
\AxiomC{$\proofvdots{\rho}$}
\noLine\UnaryInfC{$\Gamma \vdash A$}
\RightLabel{\scriptsize $\& R$}
\BinaryInfC{$\Gamma \vdash A^2$}
\DisplayProof
\end{center}
As for formulas, we write $\pi^n$ for the proof $\pi \,\&\, \dots \,\&\, \pi$, where there are $n$ copies of $\pi$.
\end{definition}

We will frequently need to discard unwanted booleans without the use of exponentials. This can be achieved with the following family of proofs.

\begin{definition}
Let $\pi$ be a proof of $\Gamma \vdash A$. We recursively define a family of proofs 
\[
{_n}\pBoolWeak(\pi, k) : \Gamma, k\, \tNBool \vdash A
\] 
for $k \geq 0$ by defining ${_n}\pBoolWeak(\pi, 0) = \pi$, and defining ${_n}\pBoolWeak(\pi, k+1)$ as the proof
\begin{center}
\AxiomC{$\proofvdots{{_n}\pBoolWeak(\pi, k)^n}$}
\noLine\UnaryInfC{$\Gamma, k\, \tNBool \vdash A^n$}
\AxiomC{}
\UnaryInfC{$A \vdash A$}
\RightLabel{\scriptsize $\multimap L$}
\BinaryInfC{$\Gamma, (k+1)\, \tNBool \vdash A$}
\DisplayProof
\end{center}
\end{definition}

\subsection{Relative step}\label{section:relstepdefn}

The goal of this section is to construct a component-wise plain proof
\[
{_h}\pRelStep : {!}\tSBool^{\otimes 2h+1} \otimes {!}\tNBool \vdash {!}\tSBool^{\otimes 2h+3} \otimes {!}\tNBool
\]
which encodes a single step transition of a Turing machine, such that the position of the head remains in the middle of the sequence of $s$-booleans representing the tape. For this to be possible, we introduce two new copies of the blank symbol, to be appended to the left end of the tape if the head moves left, and likewise for right. As an example, if the tape contents were initially
\[
\sigma_{-h} \ldots \sigma_{-1} \sigma_0 \sigma_1 \ldots \sigma_h
\]
then $\pRelStep$ would produce the sequence
\[
 \, \Box\, \Box \,\sigma_{-h} \dots \sigma_{-1} \sigma_0' \sigma_1 \dots \sigma_h
\]
if the Turing machine moved left, where $\sigma_0'$ is the newly written symbol. Note that the new symbol in the centre of the tape is $\sigma_{-1}$. If instead the Turing machine moved right, then $\pRelStep$ would produce the sequence
\[
 \sigma_{-h} \dots \sigma_{-1} \sigma_0' \sigma_1 \dots \sigma_h\, \Box\, \Box,
\]
thus causing $\sigma_1$ to become the new symbol in the centre of the tape.

\begin{lemma}
For $-h-1 \leq m \leq h+1$, there exists a proof ${^m}\pSymbol$ which computes the new symbol in position $m$, relative to the \textit{new} tape head position. \end{lemma}

\begin{proof}
In order to compute the symbol in relative position $m$, the proof ${^m}\pSymbol$ will require a copy of the symbol in position $m-1$ to be used if we move left, and the symbol in position $m+1$ if we move right. We will also need to provide each proof ${^m}\pSymbol$ with a copy of the state and the currently scanned symbol, in order to compute the direction to direction to move. Define $\pi$ as the following proof.
\begin{center}
\AxiomC{$\proofvdots{{_\delta^2}\pTrans}$}
\noLine\UnaryInfC{${\tSBool}, {\tNBool} \vdash \tBool$}

\AxiomC{$\proofvdots{\pBoolWeak({^s}\pEval, 1)}$}
\noLine\UnaryInfC{${\tSBool}, {\tSBool}, A^s \vdash A$}

\AxiomC{$\proofvdots{\pBoolWeak({^s}\pEval, 1)}$}
\noLine\UnaryInfC{${\tSBool}, {\tSBool}, A^s \vdash A$}
\RightLabel{\scriptsize exch}
\UnaryInfC{${\tSBool}, {\tSBool}, A^s \vdash A$}

\RightLabel{\scriptsize $\& R$}
\BinaryInfC{${\tSBool}, {\tSBool}, A^s \vdash A^2$}
\AxiomC{}
\UnaryInfC{$A \vdash A$}
\RightLabel{\scriptsize $\multimap L$}
\BinaryInfC{${\tSBool}, {\tSBool}, \tBool, A^s \vdash A$}
\RightLabel{\scriptsize $\multimap R$}
\UnaryInfC{${\tSBool}, {\tSBool}, \tBool \vdash \tSBool$}
\RightLabel{\scriptsize cut}
\BinaryInfC{${\tSBool}, {\tSBool}, \tSBool, \tNBool \vdash \tSBool$}
\DisplayProof
\end{center}
If $\alpha, \beta, \sigma \in \Sigma$ and $q \in Q$ then we have
\[
\dntn{\pi}(\alpha, \beta, \sigma, q) = \begin{cases} \alpha & \delta_2(\sigma, q) = 0\; (\text{left}) \\ \beta & \delta_2(\sigma, q) = 1\; (\text{right}).\end{cases}
\]
For the symbols not near the tape head ($m \neq \pm 1$) or the ends of the tape ($m \neq \pm h, \pm(h+1)$), we can simply define ${^m}\pSymbol = \pi$, with the understanding that the inputs to ${^m}\pSymbol$ will be the $s$-booleans corresponding to relative position $m-1, m+1$ and $0$ respectively. For $m = \pm h, \pm (h+1)$, we define ${^m}\pSymbol$ as the cut of the blank symbol $\underline{0}$ against $\pi$, in order to introduce new blank symbols at the ends of the tape. This produces a proof of $\tSBool, \tSBool, \tNBool \vdash \tSBool$. Finally for $m = \pm 1$, we must also compute a copy of the new symbol, which can be achieved by cutting ${_\delta^0}\pTrans$ against $\pi$, producing a proof of $\tSBool, \tNBool, \tSBool, \tSBool, \tNBool \vdash \tSBool$.
\end{proof}

\begin{proposition}
For $h \geq 0$, there is a component-wise plain proof
\[
{_h}\pRelStep : {!} \tSBool^{\otimes 2h+1} \otimes {!}\tNBool \vdash {!} \tSBool^{\otimes 2h+3} \otimes {!}\tNBool
\]
which encodes a single transition step of a given Turing machine, using relative tape coordinates.
\end{proposition}

\begin{proof}
For $-h-1 \le m \le h + 1$ we derelict the hypotheses of ${^m}\pSymbol$ and then promote the result, do the same with ${_\delta^1}\pTrans$, and then we tensor all these promoted proofs together, and perform contractions. In total, it is necessary to create two copies of each $\tSBool$ other than the current symbol, along with $2k+6$ copies of the current symbol and state. These copies are used to produce two copies of the new symbol, one copy of the new state, and $2k+3$ copies of the direction to move; one for each of the proofs ${^m}\pSymbol$.
\end{proof}

\begin{remark}\label{remark:denotation_relstep} We now compute the polynomials associated to ${_h} \pRelStep$:
\begin{align*}
\alpha^m = \sum_{i \in \Sigma} a_{i}^m \dntn{\underline{i}}, \qquad \beta = \sum_{q \in Q} b_q \dntn{\underline{q}}.
\end{align*}
We compute the value of $\dntn{{_h} \pRelStep}$ on $\vacu_{\alpha^{-h}} \otimes ... \otimes \vacu_{\alpha^{h}} \otimes \vacu_{\beta}$. The image of this element is a tensor of the form
\[
\vacu_{\theta^{-h-1}} \otimes ... \otimes \vacu_{\theta^{h+1}} \otimes \vacu_{\mu}
\]
where the individual vectors $\theta^m, \mu$ are described as follows. Let $(\hat{\sigma}_i^q, \hat{q}_i^q, \hat{d}_i^q) = \delta(i,q)$. For $m \neq \pm 1, \pm h, \pm(h+1)$ the vector $\theta^m$ is given by
\begin{align*}
\dntn{{^m}\pSymbol}&(\alpha^{m-1} \otimes \alpha^{m+1} \otimes \alpha^0 \otimes \beta) 
\\&= \sum_{i \in \Sigma} \sum_{j \in \Sigma} \sum_{k \in \Sigma} \sum_{q \in Q} a_{i}^{m-1} a_{j}^{m+1} a_{k}^{0} b_q (\delta_{\hat{d}_k^q = 0}\dntn{\underline{i}} + \delta_{\hat{d}_k^q = 1}\dntn{\underline{j}})
\\&= \left(\sum_{k \in \Sigma} \sum_{q \in Q} a_{k}^{0} b_q \delta_{\hat{d}_k^q = 0}\right)\left(\sum_{i \in \Sigma}a_{i}^{m-1}\dntn{\underline{i}}\right) \left(\sum_{j \in \Sigma} a_{j}^{m+1}\right) 
\\&\qquad+ \left(\sum_{k \in \Sigma} \sum_{q \in Q} a_{k}^{0} b_q \delta_{\hat{d}_k^q = 1}\right)\left(\sum_{i \in \Sigma}a_{i}^{m-1}\right) \left(\sum_{j \in \Sigma} a_{j}^{m+1}\dntn{\underline{j}}\right).
\end{align*}
For $m = h, h+1$ the vector $\theta^m$ is given by
\begin{align*}
\dntn{{^m}\pSymbol}&(\alpha^{m-1} \otimes \alpha^0 \otimes \beta) 
\\&= \sum_{i \in \Sigma} \sum_{k \in \Sigma} \sum_{q \in Q} a_{i}^{m-1} a_{k}^{0} b_q (\delta_{\hat{d}_k^q = 0}\dntn{\underline{i}} + \delta_{\hat{d}_k^q = 1}\dntn{\underline{0}})
\\&= \left(\sum_{k \in \Sigma} \sum_{q \in Q} a_{k}^{0} b_q \delta_{\hat{d}_k^q = 0}\right)\left(\sum_{i \in \Sigma}a_{i}^{m-1}\dntn{\underline{i}}\right)
\\&\qquad+ \left(\sum_{k \in \Sigma} \sum_{q \in Q} a_{k}^{0} b_q \delta_{\hat{d}_k^q = 1}\right)\left(\sum_{i \in \Sigma}a_{i}^{m-1}\right)\dntn{\underline{0}}.
\end{align*}
For $m = -h, -h-1$ the vector $\theta^m$ is given by
\begin{align*}
\dntn{{^m}\pSymbol}&(\alpha^{m+1} \otimes \alpha^0 \otimes \beta) 
\\&= \sum_{i \in \Sigma} \sum_{k \in \Sigma} \sum_{q \in Q} a_{i}^{m+1} a_{k}^{0} b_q (\delta_{\hat{d}_k^q = 0}\dntn{\underline{0}} + \delta_{\hat{d}_k^q = 1}\dntn{\underline{i}})
\\&= \left(\sum_{k \in \Sigma} \sum_{q \in Q} a_{k}^{0} b_q \delta_{\hat{d}_k^q = 0}\right)\left(\sum_{i \in \Sigma}a_{i}^{m+1}\right)\dntn{\underline{0}}
\\&\qquad+ \left(\sum_{k \in \Sigma} \sum_{q \in Q} a_{k}^{0} b_q \delta_{\hat{d}_k^q = 1}\right)\left(\sum_{i \in \Sigma}a_{i}^{m+1}\dntn{\underline{i}}\right).
\end{align*}
Finally $\theta^1$ is
\begin{align*}
\dntn{{^1}\pSymbol}&(\alpha^{0} \otimes \beta \otimes \alpha^{2} \otimes \alpha^0  \otimes \beta) 
\\&= \sum_{i \in \Sigma} \sum_{p \in Q} \sum_{j \in \Sigma} \sum_{k \in \Sigma}  \sum_{q \in Q} a_{i}^{0} b_p a_{j}^{2} a_{k}^{0} b_q (\delta_{\hat{d}_k^q = 0}\dntn{\underline{\hat{\sigma}_i^p}} + \delta_{\hat{d}_k^q = 1}\dntn{\underline{j}})
\\&= \left(\sum_{k \in \Sigma} \sum_{q \in Q} a_{k}^{0} b_q \delta_{\hat{d}_k^q = 0}\right)\left(\sum_{i \in \Sigma} \sum_{p \in Q} a_{i}^{0} b_p\dntn{\underline{\hat{\sigma}_i^p}}\right)\left(\sum_{j \in \Sigma} a_{j}^{2}\right)
\\&\qquad+ \left(\sum_{k \in \Sigma} \sum_{q \in Q} a_{k}^{0} b_q \delta_{\hat{d}_k^q = 1}\right)\left(\sum_{i \in \Sigma} a_{i}^{0} \right)\left(\sum_{p \in Q} b_p\right)\left(\sum_{j \in \Sigma} a_{j}^{2}\dntn{\underline{j}}\right).
\end{align*}
and $\theta^{-1}$ is
\begin{align*}
\dntn{{^{-1}}\pSymbol}&(\alpha^{-2} \otimes \alpha^{0} \otimes \beta \otimes \alpha^0  \otimes \beta) 
\\&= \sum_{i \in \Sigma} \sum_{j \in \Sigma} \sum_{p \in Q} \sum_{k \in \Sigma}  \sum_{q \in Q} a_{i}^{-2} a_{j}^{0} b_p a_{k}^{0} b_q (\delta_{\hat{d}_k^q = 0}\dntn{\underline{i}} + \delta_{\hat{d}_k^q = 1}\dntn{\underline{\hat{\sigma}_j^p}})
\\&= \left(\sum_{k \in \Sigma} \sum_{q \in Q} a_{k}^{0} b_q \delta_{\hat{d}_k^q = 0}\right)\left(\sum_{i \in \Sigma} a_{i}^{-2}\dntn{\underline{i}}\right)\left(\sum_{j \in \Sigma} a_{j}^{0} \right)\left(\sum_{p \in Q} b_p\right)
\\&\qquad+ \left(\sum_{k \in \Sigma} \sum_{q \in Q} a_{k}^{0} b_q \delta_{\hat{d}_k^q = 1}\right)\left(\sum_{i \in \Sigma} a_{i}^{-2}\right)\left(\sum_{j \in \Sigma} \sum_{p \in Q} a_{j}^{0} b_p\dntn{\underline{\hat{\sigma}_j^p}}\right).
\end{align*}
The vector $\mu$ is computed by
\begin{align*}
\dntn{{_\delta^1}\pTrans}&(\alpha^{0} \otimes \beta) = \sum_{k \in \Sigma} \sum_{q \in Q} a_{k}^{0} b_q \dntn{\underline{\hat{q}_k^q}}.
\end{align*}
\end{remark}

\begin{remark}
We have now given two encodings of the step function of a Turing machine which use booleans to encode the contents of the tape squares, and which index these tape squares relative to the head, namely $\pBoolstep$ (in Section \ref{section:actionofstep}) and $\pRelStep$. These are not proofs of the same sequents but there is a natural way to compare them when $s = 2$, and the result of this comparison is negative: the encodings are genuinely different.

Recall that ${^p}\pBoolstep_A^{a,b,c,d}$ is a proof of
\be
a \,{!} \tBool_B, b\, {!} \tBool_B, {!} {_n} \tBool_B \vdash {!} \tBool_A^{\otimes c } \otimes {!} \tBool_A^{\otimes d} \otimes {!} {_n} \tBool_A
\ee
where $B = A^{3^{p+e+1}}$ with $e = \max\{c, d\}$. In general, the diagram
\[
\xymatrix@C+7pc@R+2pc{
{!} \dntn{\tBool_B}^{\otimes 2h+1} \otimes {!} \dntn{{}_n \tBool_B} \ar[d]^-{{!}\dntn{\pBooltype}^{\otimes 2h+1} \otimes {!}\dntn{{}_n \pBooltype}}\ar[r]^-{\dntn{{^1}\pBoolStep^{h+1,h,h+2,h+1}}} & {!} \dntn{\tBool_A}^{\otimes 2h+3} \otimes {!} \dntn{{}_n \tBool_A}\\
{!} \dntn{\tBool_A}^{\otimes 2h+1} \otimes {!} \dntn{{}_n \tBool_A} \ar[r]_-{ \dntn{{_h}\pRelStep}} & {!} \dntn{ \tBool_A }^{\otimes 2h + 3} \otimes {!} \dntn{{}_n \tBool_A} \ar[u]_-{=}
}
\]
\emph{does not} commute. However, the two encodings do give rise to the same naive probabilistic extension of the step function; see \cite[Appendix D]{clift_murfet3}.
\end{remark}

\subsection{Absolute step}\label{section:absolutestep}

Suppose we are confined to a region of the tape which is $h$ symbols long, encoded by a sequence of $\tSBool$'s which appear in the same order as they occur on the tape. The goal of this section is to construct a component-wise plain proof 
\[
{_h}\pAbsStep : {!} \tSBool^{\otimes h} \otimes {!}\tNBool \otimes {!}\tHBool \vdash {!} \tSBool^{\otimes h} \otimes {!}\tNBool \otimes {!}\tHBool
\]
which encodes a single step transition of a Turing machine, where the purpose of the $h$-boolean is to keep track of the head position, and thus it decrements if we move left, and increments if we move right. The valid positions are $\{0,\ldots,h-1\}$ which we identify with $\mathbb{Z}/h\mathbb{Z}$. All positions are read modulo $h$, and so if the head is currently in position $h-1$ and it moves right, it moves to position $0$, and similarly if the head is at position $0$ and moves left, it moves to position $h-1$.

\begin{lemma}
For $0 \leq m \leq h-1$, there exists a proof 
\[
{^m}\pSymbol : \tSBool, \tNBool, \tHBool \vdash \tSBool
\]
which encodes the function
\[
(\sigma, q, i) \mapsto \begin{cases}\sigma & i \neq m \\ \delta_0(\sigma, q) & i = m.\end{cases}
\]
\end{lemma}

\begin{proof}

Let $\pi : \tSBool, \tNBool, A^s \vdash A$ be the proof ${_n}\pBoolWeak({^s}\pEval, 1)$. Then ${^m}\pSymbol$ is the following proof.
\begin{center}
\AxiomC{$\proofvdots{\pi^{m} \,\&\, {_\delta^0}\pTrans\,\&\, \pi^{h-m-1}}$}
\noLine\UnaryInfC{$\tSBool, \tNBool, A^s \vdash A^h$}
\AxiomC{}
\UnaryInfC{$A \vdash A$}
\RightLabel{\scriptsize $\multimap L$}
\BinaryInfC{$\tSBool, \tNBool, \tHBool, A^s \vdash A$}
\RightLabel{\scriptsize $\multimap R$}
\UnaryInfC{$\tSBool, \tNBool, \tHBool \vdash \tSBool$}
\DisplayProof
\end{center}
\end{proof}

\begin{lemma} \label{lemma: state for cstep3}
There exists a proof
\[
\pState : h\, \tSBool, \tNBool, \tHBool \vdash \tNBool
\]
which encodes the function
\[
(\sigma_0, ..., \sigma_{h-1}, q, i) \mapsto \delta_1(\sigma_i, q).
\]
\end{lemma}
\begin{proof}
For $0 \leq i \leq h-1$, define $\pi_i$ be the proof  as the proof \[{_s}\pBoolWeak({^1_\delta}\pTrans, h-1) : h\, \tSBool, \tNBool, A^n \vdash A,\] where the original $\tSBool$ from the proof ${^1_\delta}\pTrans$ is in the $i$th position on the left of the turnstile in $\pi_i$. Then $\pState$ is the following proof.
\begin{center}
\AxiomC{$\proofvdots{\pi_0 \,\&\, \dots \,\&\, \pi_{h-1}}$}
\noLine\UnaryInfC{$h\,  \tSBool, \tNBool, A^n \vdash A^h$}
\AxiomC{}
\UnaryInfC{$A \vdash A$}
\RightLabel{\scriptsize $\multimap L$}
\BinaryInfC{$h\, \tSBool, \tNBool, \tHBool, A^n \vdash A$}
\RightLabel{\scriptsize $\multimap R$}
\UnaryInfC{$h\, \tSBool, \tNBool, \tHBool \vdash \tNBool$}
\DisplayProof
\end{center}
\end{proof}

\begin{lemma}
There exists a proof
\[
\pTapehead : h\, \tSBool, \tNBool, \tHBool, \tHBool \vdash \tHBool
\]
which encodes the function
\[
(\sigma_0, ..., \sigma_{h-1}, q, i, j) \mapsto \begin{cases} i-1 & \delta_2(\sigma_j, q) = \text{left} \\ i+1 & \delta_2(\sigma_j, q) = \text{right}.\end{cases}
\]
\end{lemma}
\begin{proof}
Let $\tau_L: \tHBool, A^h \vdash A$ be the proof which converts the $h$-boolean $\underline{i}$ into the $h$-boolean $\underline{i-1}$ modulo $h$, defined as follows:
\begin{center}
\AxiomC{}
\UnaryInfC{$A \vdash A$}
\RightLabel{\scriptsize $\&L_{h-1}$}
\UnaryInfC{$A^h \vdash A$}
\AxiomC{}
\UnaryInfC{$A \vdash A$}
\RightLabel{\scriptsize $\&L_0$}
\UnaryInfC{$A^h \vdash A$}
\AxiomC{\dots}
\AxiomC{}
\UnaryInfC{$A \vdash A$}
\RightLabel{\scriptsize $\&L_{h-2}$}
\UnaryInfC{$A^h \vdash A$}
\RightLabel{\scriptsize $\&R$}
\QuaternaryInfC{$A^h \vdash A^h$}
\AxiomC{}
\UnaryInfC{$A \vdash A$}
\RightLabel{\scriptsize $\multimap L$}
\BinaryInfC{$\tHBool, A^h \vdash A$}
\DisplayProof
\end{center}
Similarly define a proof $\tau_R$ which encodes $\underline{i} \mapsto \underline{i+1}$ modulo $h$. 

Given $0 \leq i \leq h-1$, let $\rho_i$ be the following proof, where as in Lemma \ref{lemma: state for cstep3} the $i$ indicates which $\tSBool$ in the leftmost branch is the original copy from the proof ${^2}\pTrans$.
\begin{center}
\AxiomC{$\proofvdots{{_s}\pBoolWeak({^2_\delta}\pTrans, h-1)}$}
\noLine\UnaryInfC{$h\, \tSBool, \tNBool, A^2 \vdash A$}
\RightLabel{\scriptsize $\multimap R$}
\UnaryInfC{$h\, \tSBool, \tNBool \vdash \tBool$}
\AxiomC{$\proofvdots{\tau_L \,\&\, \tau_R}$}
\noLine\UnaryInfC{$\tHBool, A^h\vdash A^2$}
\AxiomC{}
\UnaryInfC{$A \vdash A$}
\RightLabel{\scriptsize $\multimap L$}
\BinaryInfC{$\tHBool, \tBool, A^h\vdash A$}
\RightLabel{\scriptsize cut}
\BinaryInfC{$h\, \tSBool, \tNBool, \tHBool, A^h \vdash A$}
\DisplayProof
\end{center}

Then $\pTapehead$ is:
\begin{center}
\AxiomC{$\proofvdots{\rho_0 \, \&\, \dots \, \&\, \rho_{h-1}}$}
\noLine\UnaryInfC{$h\,  \tSBool, \tNBool, \tHBool, A^h \vdash A^h$}
\AxiomC{}
\UnaryInfC{$A \vdash A$}
\RightLabel{\scriptsize $\multimap L$}
\BinaryInfC{$h\, \tSBool, \tNBool, \tHBool, \tHBool, A^h \vdash A$}
\RightLabel{\scriptsize $\multimap R$}
\UnaryInfC{$h\, \tSBool, \tNBool, \tHBool, \tHBool \vdash \tHBool$}
\DisplayProof
\end{center}
\end{proof}

\begin{proposition}
There is a component-wise plain proof
\[
{_h}\pAbsStep : {!} \tSBool^{\otimes h} \otimes {!}\tNBool \otimes {!}\tHBool \vdash {!} \tSBool^{\otimes h} \otimes {!}\tNBool \otimes {!}\tHBool
\]
which encodes a single transition step of a given Turing machine, using absolute tape coordinates.
\end{proposition}

\begin{proof}
Clear from the above. In total, the contraction steps at the bottom of the tree produce $3$ copies of each tape $\tSBool$, $h+2$ copies of the state $\tNBool$, and $h+3$ copies of the $\tHBool$ which keeps track of the head position.
\end{proof}

\begin{remark}
One interesting feature of the absolute encoding is that it can be iterated for a variable number of steps through the use of an $\tInt$, since the denotation of ${_h}\pAbsStep$ is an endomorphism. Note that this is impossible for the relative encoding, since the type changes after each step due to the growth of the tape.
\end{remark}

\begin{remark}\label{remark:denotation_absstep}
We now compute the denotation of ${_h} \pAbsStep$. Let $H = \{0, ..., h-1\}$, and define
\begin{align*}
\alpha^m = \sum_{i_m \in \Sigma} a_{i_m}^m \dntn{\underline{i_m}}, \qquad \beta = \sum_{q \in Q} b_q \dntn{\underline{q}}, \qquad \gamma = \sum_{k \in H} c_k \dntn{\underline{k}}.
\end{align*}
We compute the value of $\dntn{{_h}\pAbsStep}$ on $\vacu_{\alpha^0} \otimes ... \otimes \vacu_{\alpha^{h-1}} \otimes \vacu_{\beta} \otimes \vacu_{\gamma}$. Let $(\hat{\sigma}_i^q, \hat{q}_i^q, \hat{d}_i^q) = \delta(i,q)$. Note that:
\begin{align*}
\dntn{{^m}\pSymbol}&(\alpha^m \otimes \beta \otimes \gamma) 
\\&= \sum_{i_m \in \Sigma} \sum_{q \in Q} a_{i_m}^m b_q\left(c_m \dntn{\underline{\hat{\sigma}_{i_m}^q}} + \sum_{\substack{k \in H \\ k \neq m}}  c_k \dntn{\underline{i_m}}\right),
\\
\dntn{\pState}&(\alpha^0 \otimes ... \otimes \alpha^{h-1} \otimes \beta \otimes \gamma) 
\\&= \sum_{i_0 \in \Sigma} \cdots \sum_{i_{h-1} \in \Sigma} \sum_{q \in Q} \sum_{k \in H} a_{i_0} \dots a_{i_{h-1}} b_q c_k\dntn{\underline{\hat{q}_{i_k}^q}}
\\&= \sum_{k \in H} c_k\left[\left(\sum_{i_{k} \in \Sigma}\sum_{q \in Q} a_{i_{k}}^{k} b_q \dntn{\underline{\hat{q}_{i_k}^q}}\right) \prod_{\substack{m \in H \\m \neq k}} \left(\sum_{i_m \in \Sigma} a_{i_m}^m\right)\right], \text{ and }
\\
\dntn{\pTapehead}&(\alpha^0 \otimes ... \otimes \alpha^{h-1} \otimes \beta \otimes \gamma^{\otimes 2}) 
\\&= \sum_{i_0 \in \Sigma} \cdots \sum_{i_{h-1} \in \Sigma} \sum_{q \in Q} \sum_{k \in H} \sum_{l \in H} a_{i_0} \dots a_{i_{h-1}} b_q c_k c_l \left(\delta_{\hat{d}_{i_l}^q = 0}\dntn{\underline{k-1}} + \delta_{\hat{d}_{i_l}^q = 1}\dntn{\underline{k+1}}\right)
\\&= \sum_{l \in H} c_l\left[\left(\sum_{i_{l} \in \Sigma}\sum_{q \in Q}\sum_{k \in H} a_{i_{l}}^{l} b_q c_k \left(\delta_{\hat{d}_{i_l}^q = 0}\dntn{\underline{k-1}} + \delta_{\hat{d}_{i_l}^q = 1}\dntn{\underline{k+1}}\right)\right) \prod_{\substack{m \in H \\m \neq l}} \left(\sum_{i_m \in \Sigma} a_{i_m}^m\right)\right],
\end{align*}
where $k\pm 1$ is computed modulo $h$.
\end{remark}

\appendix

\section{Extensions of the encoding}

In this section we consider various extensions of the base model of Turing machines. They are multi-tape machines (Section \ref{section:multi_tape}), machines with extended tape alphabets (Section \ref{section:tape_alphabet}) and machines that are able to leave the head position unchanged (Section \ref{section:headstill}). We give each of these as an extension of the original encoding, but they may also be combined to give, for example, a multi-tape Turing machine with extended tape alphabet.

\subsection{Multiple tapes}\label{section:multi_tape}

In this section we consider the modifications that need to be made in order to represent a Turing machine with $k \ge 1$ tapes. All of our tapes are both read and write. To begin with, the transition function is now
\[
\delta: \Sigma^k \times Q \lto \Sigma^{k} \times Q \times \{\text{left, right}\}^k
\]
where an input-output pair
\[
\delta( (\sigma_1,\ldots,\sigma_k), q ) = ((\sigma'_1,\ldots,\sigma'_k), q', (d_1,\ldots,d_k))
\]
represents the machine reading $\sigma_1,\ldots,\sigma_k$ on tapes $1$ through $k$ (in that order) and being in state $q$, and then writing $\sigma'_1,\ldots,\sigma'_k$ onto the tapes $1$ through $k$, taking the new state $q'$ and moving in the direction $d_i$ on tape $i$. Observe that all of the tapes are read-write; if we want the first tape to be read-only as in \cite[\S 1.2]{arorabarak} this is achieved by restricting the valid transition functions $\delta$ to be such that $\sigma'_1 = \sigma_1$ for all inputs.

The \emph{configuration} of the multi-tape Turing machine is a tuple
\[
\langle \underline{S}, \underline{T}, q \rangle \in (\Sigma^*)^k \times (\Sigma^*)^k \times Q\,.
\]
The type of $k$-\emph{tape Turing configurations} on $A$ is
\[
\tTur_A^k = k \,\big( {!} \tBint_A \otimes {!} \tBint_A \big) \otimes {!} {}_n \tBool_A\,.
\]
We need the following generalisation of Lemma \ref{encoding functions of n bools}:

\begin{lemma}\label{encoding functions of n bools mod}
Given $k \ge 1$ and any function $f: \{0, ..., n-1\}^k \to \{0, ..., m-1\}$, there exists a proof $F$ of $k\,\tNBool_A \vdash \tMBool_A$ which encodes $f$.
\end{lemma} 

\begin{proof}
The proof is by induction on $k$, with $k = 1$ being Lemma \ref{encoding functions of n bools}. Assuming the lemma holds for $k-1$ we for each $0 \le z \le n-1$ let $F_z$ be the proof of $(k-1)\,\tNBool_A \vdash \tMBool_A$ encoding the function
\[
f(-, z): \{0, ..., n-1\}^{k-1} \to \{0, ..., m-1\}\,.
\]
Then the proof
\begin{center}
\AxiomC{$\proofvdots{F_1}$}
\noLine\UnaryInfC{$A^m, (k-1)\, \tNBool_A \vdash A$}
\AxiomC{}
\noLine\UnaryInfC{$...$}
\AxiomC{$\proofvdots{F_n}$}
\noLine\UnaryInfC{$A^m, (k-1)\, \tNBool_A \vdash A$}

\RightLabel{\scriptsize $\&R$}
\doubleLine\TrinaryInfC{$A^m, (k-1)\, \tNBool_A \vdash A^n$}

\AxiomC{}
\UnaryInfC{$A \vdash A$}

\RightLabel{\scriptsize $\multimap L$}
\BinaryInfC{$A^m, (k-1)\, \tNBool_A, \tNBool_A \vdash A$}
\RightLabel{\scriptsize $\multimap R$}
\UnaryInfC{$(k-1)\, \tNBool_A, \tNBool_A \vdash \tMBool_A$}

\DisplayProof
\end{center}
encodes $f$.
\end{proof}

Let us now explain how to modify the earlier encoding of the step function to construct a component-wise plain proof
\[
{}_\delta \pStep_A: \tTur_{A^3}^k \vdash \tTur_A^k
\]
which encodes the step function of the $k$-tape Turing machine. Let us proceed ingredient-by-ingredient, beginning with Proposition \ref{encoding transition functions} where we defined the encoding of the components of the transition function. For the multi-tape machine $\delta$ has $2k + 1$ components rather than $3$, but using the same ideas we may define encodings $\{ {^i_\delta}\pTrans_A \}_{i=0}^{2k}$ which have on the left hand side of the turnstile the sequence of inputs $k\, \tBool_A, {}_n \tBool_A$. We use the old proofs ${^0}\pRecomb_A, {^2}\pRecomb_A$ un-modified in what follows.

As in Proposition \ref{prop: left right and state} we have to define for $1 \le r \le k$ proofs
\begin{gather*}
{^r_\delta}\pLeft_A:  \; \tBint_{A^3}, 2k\,\tBint_{A^3}, \tBint_{A^3}, 2\,\tNBool_{A^3} \vdash \tBint_A \\
{^r_\delta}\pRight_A: \; 2k\,\tBint_{A^3}, 2\,\tBint_{A^3}, 2\,\tNBool_{A^3} \vdash \tBint_A \\
{_\delta}\pState_A: \; k\,\tBint_{A^3},\,\tNBool_{A^3} \vdash \tNBool_A 
\end{gather*}
which respectively compute the left hand part of the $r$th tape, the right hand part of the $r$th tape, and the new state. To explain the necessary modifications to the proof trees given earlier, we use the informal calculations given in the proof of Proposition \ref{prop: left right and state}. To this end, suppose the configuration is
\[
\Big\langle \big( S_i \sigma_i \big)_{i=1}^k, \big( T_i \tau_i )_{i=1}^k, q \Big\rangle
\]
and that
\[
\delta\big( (\sigma_1,\ldots,\sigma_k), q \big) = \big( (\sigma'_1,\ldots,\sigma'_k), q', (d_1,\ldots,d_k) \big)
\]
We write
\[
(S'_i, T'_i) = 
\begin{cases}
(S_i, T_i \tau_i\sigma'_i) & d_i = 0 \text{ (left)}
\\(S_i\sigma'_i\tau_i, T) & d_i = 1 \text{ (right)}
\end{cases}
\]
and using this notation define 
\begin{align*}
{^r_\delta}\pLeft_A \text{ is}:
&\phantom{\xmapsto{\makebox[0.6cm]{}}} S_r \sigma_r \otimes \otimes_{i=1}^k (S_i \sigma_i)^{\otimes 2} \otimes T_r \tau_r \otimes q^{\otimes 2}
\\&\xmapsto{\makebox[0.5cm]{}}   S_r \otimes \otimes_{i=1}^k \sigma_i^{\otimes 2} \otimes \tau_r \otimes q^{\otimes 2}
& ( \pTail_A \otimes \otimes_{i=1}^k( \pHead_A^{\otimes 2} ) \otimes \pHead_A \otimes {_n}\pBooltype_A^{\otimes 2})
\\&\xmapsto{\makebox[0.5cm]{}}   S_r \otimes \tau_r \otimes \big( \!\otimes_{i=1}^k \sigma_i \otimes q\big)^{\otimes 2}
& (\text{exchange})
\\&\xmapsto{\makebox[0.5cm]{}}   S_r \otimes \tau_r \otimes \sigma'_r \otimes d_r
& (\id^{\otimes 2} \otimes \prescript{r-1}{\delta}\pTrans_A \otimes 
\prescript{k+r}{\delta}\pTrans_A)
\\&\xmapsto{\makebox[0.5cm]{}}   S'_r
& ({^0}\pRecomb_A)\end{align*}
\begin{align*}
{^r_\delta}\pRight_A \text{ is}: 
&\phantom{\xmapsto{\makebox[0.6cm]{}}} \otimes_{i=1}^k (S_i \sigma_i)^{\otimes 2} \otimes (T_r \tau_r)^{\otimes 2} \otimes q^{\otimes 2}
\\&\xmapsto{\makebox[0.5cm]{}}   \otimes_{i=1}^k \sigma_i^{\otimes 2} \otimes \tau_r \otimes T_r \otimes q^{\otimes 2}
& (\otimes_{i=1}^k \pHead_A^{\otimes 2} \otimes \pHead_A \otimes \pTail_A \otimes {_n}\pBooltype_A^{\otimes 2})
\\&\xmapsto{\makebox[0.5cm]{}}   T_r \otimes \tau_r \otimes \big(\!\otimes_{i=1}^k \sigma_i \otimes q\big)^{\otimes 2}
& (\text{exchange})
\\&\xmapsto{\makebox[0.5cm]{}}   T_r \otimes \tau_r \otimes \sigma'_r \otimes d_r
& (\id^{\otimes 2} \otimes \prescript{r-1}{\delta}\pTrans_A \otimes \prescript{k+r}{\delta}\pTrans_A)
\\&\xmapsto{\makebox[0.5cm]{}}   T'_r
& ({^1}\pRecomb_A) \\ \\
{_\delta}\pState_A \text{ is}:
&\phantom{\xmapsto{\makebox[0.6cm]{}}} \otimes_{i=1}^k (S_i \sigma_i) \otimes q
\\&\xmapsto{\makebox[0.5cm]{}}   \otimes_{i=1}^k \sigma_i \otimes q
& (\otimes_{i=1}^k \pHead_A \otimes {_n}\pBooltype_A)
\\&\xmapsto{\makebox[0.5cm]{}}   q'.
& ({^k_\delta}\pTrans_A)
\end{align*}
We define ${}_\delta \pStep$ by derelicting and promoting the proofs
\[
{^1_\delta}\pLeft_A,\ldots,{^r_\delta}\pLeft_A, {^1_\delta}\pRight_A,\ldots, {^r_\delta}\pRight_A, {_\delta}\pState_A
\]
then tensoring them together and performing contractions. The contractions to be performed are also contained implicitly in the above calculations: for example, the ordered inputs to ${}^r_{\delta} \pLeft_A$ are to be identified, by such contractions, with (1) the left hand part of the $r$th tape (2) two copies of the left part of each of the $k$ tapes (3) a copy of the right part of the $r$th tape (4) and two copies of the state.

\subsection{Extended tape alphabet}\label{section:tape_alphabet}

In Section \ref{section: turing machines} we have encoded only Turing machines with tape alphabet $\Sigma = \{ 0, 1 \}$, and in this section we explain the routine modifications to be made for arbitrary $\Sigma$. We begin by generalising the presentation of \cite[\S 3.2]{clift_murfet} following \cite[\S 2.5.3]{girard_complexity} to give the type of lists. Throughout we fix an integer $s \ge 1$.

\begin{definition}
The type of \emph{$s$-lists on $A$} is defined inductively by ${_1} \tList_A = \tInt_A$ and
\[
{_{s+1}} \tList_A = {!}(A \multimap A) \multimap {_s} \tList_A\,.
\]
For $0 \le i \le s - 1$ the proof $\underline{i}_A$ of ${}_s \tList_A$ is the proof
\begin{center}
\AxiomC{$A \multimap A \vdash A \multimap A$}
\RightLabel{\scriptsize der}
\UnaryInfC{${!}(A \multimap A) \vdash A \multimap A$}
\doubleLine\RightLabel{\scriptsize weak}
\UnaryInfC{$s\,{!}(A \multimap A) \vdash A \multimap A$}
\DisplayProof
\end{center}
where in the final line the $i$th copy (reading from the left and starting from $0$) of ${!}(A \multimap A)$ is the one from the previous rule, and the rest are the result of weakenings. The proof $\emptyset$ is the result of applying $s$ weakenings and then $\multimap R$ rules.
\end{definition}

It will be convenient in this section to use a shorthand for proof trees that we now introduce. Suppose given a sequence of proofs $( \pi_i : A_i \vdash B_i )_{i=1}^s$ and a proof $\tau : C \vdash D$. By applying a $\multimap L$ rule to the pair $\pi_s, \tau$ we obtain a proof of
\[
A_s, B_s \multimap C \vdash D\,.
\]
Applying a $\multimap L$ rule again to $\pi_{s-1}$ and the proof above, we get a proof of
\[
A_{s-1}, B_{s-1} \multimap ( B_s \multimap C ) \vdash D\,.
\]
The result of continuing this inductive process with all the $\pi_i$ in reverse order using $\multimap L$ rules will be denoted as a proof tree by
\begin{center}
\AxiomC{$\proofvdots{\pi_1}$}
\noLine\UnaryInfC{$A_1 \vdash B_1$}
\AxiomC{}
\noLine\UnaryInfC{$\Ddots$}
\AxiomC{$\proofvdots{\pi_s}$}
\noLine\UnaryInfC{$A_s \vdash B_s$}
\AxiomC{$\proofvdots{\tau}$}
\noLine\UnaryInfC{$C \vdash D$}
\RightLabel{\scriptsize $\multimap L$}
\BinaryInfC{$A_s, B_s \multimap C \vdash D$}
\doubleLine\RightLabel{\scriptsize $\multimap L$}
\TrinaryInfC{$A_1, \ldots, A_s, B_1 \multimap ( \cdots \multimap (B_s \multimap C)\cdots) \vdash D$}
\DisplayProof
\end{center}
For the rest of this section $A$ is fixed and we write $\underline{S}$ for $\underline{S}_A$. We write $E = A \multimap A$.

\begin{definition} \label{defn:listconcat}
The proof $\pConcat_A$ for lists is defined to be
\begin{center}
\AxiomC{${!} E_1 \vdash {!}E_1$}
\AxiomC{}
\noLine\UnaryInfC{$\Ddots$}
\AxiomC{${!} E_{s} \vdash {!} E_{s}$}
\AxiomC{${!} E_1 \vdash {!} E_1$}
\AxiomC{}
\noLine\UnaryInfC{$\Ddots$}
\AxiomC{${!} E_s \vdash {!} E_s$}
\AxiomC{$\proofvdots{\pComp_A}$}
\noLine\UnaryInfC{$E, E \vdash E$}
\RightLabel{\scriptsize $\multimap L$}
\BinaryInfC{${!} E_s, E, {!} E_s \multimap E \vdash E$}
\RightLabel{\scriptsize $\multimap L$}
\doubleLine\TrinaryInfC{${!} E_1, \ldots, {!} E_s, E, {}_s \tList_A \vdash E$}
\RightLabel{\scriptsize $\multimap L$}
\BinaryInfC{${!} E_1, \ldots, {!}E_s, {!}E_s, {!}E_{s} \multimap E, {}_s \tList_A \vdash E$}
\RightLabel{\scriptsize $\multimap L$}
\doubleLine\TrinaryInfC{${!} E_1, \ldots, {!} E_s, {!} E_1, \ldots, {!} E_s, {}_s \tList_A, {}_s \tList_A \vdash E$}
\RightLabel{\scriptsize ctr}
\doubleLine\UnaryInfC{${!} E_1, \ldots, {!} E_s, {}_s \tList_A, {}_s\tList_A \vdash E$}
\RightLabel{\scriptsize $\multimap R$}
\doubleLine\UnaryInfC{${}_s\tList_A, {}_s\tList_A \vdash {}_s\tList_{A}$}
\DisplayProof
\end{center}
where $E_1,\ldots,E_s$ label copies of $E$ and in the final $\multimap R$ rules, the copies of $E$ are moved across the turnstile in \emph{decreasing} order: that is, the first $\multimap R$ rule applies to $E_s$, the second to $E_{s-1}$, and so on. For $s = 2$ this agrees with the proof $\pConcat_A$ for $\tBint_A$.
\end{definition}

\begin{definition} Given a nonempty sequence $S \in \{0,\ldots,s-1\}^*$ we define the proof $\underline{S}_A$ of ${}_s \tList_A$ by setting $S = S' \sigma$ with $\sigma \in \{0,\ldots,s-1\}$ and
\[
\underline{S}_A = \pConcat_A( \underline{S'}_A, \underline{\sigma}_A )\,.
\]
For $s = 2$ this agrees up to cut-elimination with the proof $\underline{S}_A$ of \cite[\S 3.2]{clift_murfet}.
\end{definition}

\begin{proposition} \label{decomposing bints headlist}
There exists a proof $\pHead_A$ of ${}_s \tList_{A^{s+1}} \vdash {}_s \tBool_A$ which encodes the function $\dntn{\underline{S \sigma}_{A^{s+1}}} \mapsto \dntn{\underline{\sigma}_A}$.
\end{proposition}
\begin{proof}
For $0 \le i \le s - 1$ let $\pi_i$ be the proof of $A^{s+1} \vdash A^{s+1}$ with
\[
\pi_i( a_0, \ldots, a_{s-1}, z ) = ( a_0, \ldots, a_{s-1}, a_i )\,.
\]
Let $\rho$ be the proof of $A^s \vdash A^{s+1}$ with
\[
\rho( a_0,\ldots,a_{s-1} ) = (a_0, \ldots, a_{s-1}, a_0)
\]
Define by $\pHead_A$ the following proof:
\begin{center}
\AxiomC{$\proofvdots{\pi_{0}}$}
\noLine\UnaryInfC{$A^{s+1} \vdash A^{s+1}$}
\RightLabel{\scriptsize $\multimap R$}
\UnaryInfC{$\vdash A^{s+1} \multimap A^{s+1}$}
\RightLabel{\scriptsize prom}
\UnaryInfC{$\vdash {!}(A^{s+1} \multimap A^{s+1})$}

\AxiomC{}
\noLine\UnaryInfC{$\Ddots$}

\AxiomC{$\proofvdots{\pi_{s-1}}$}
\noLine\UnaryInfC{$A^{s+1} \vdash A^{s+1}$}
\RightLabel{\scriptsize $\multimap R$}
\UnaryInfC{$\vdash A^{s+1} \multimap A^{s+1}$}
\RightLabel{\scriptsize prom}
\UnaryInfC{$\vdash {!}(A^{s+1} \multimap A^{s+1})$}

\AxiomC{$\proofvdots{\rho}$}
\noLine\UnaryInfC{$A^s \vdash A^{s+1}$}

\AxiomC{}
\UnaryInfC{$A \vdash A$}
\RightLabel{\scriptsize $\& L_s$}
\UnaryInfC{$A^{s+1} \vdash A$}

\RightLabel{\scriptsize $\multimap L$}
\BinaryInfC{$A^{s+1} \multimap A^{s+1}, A^s \vdash A$}
\RightLabel{\scriptsize $\multimap R$}
\UnaryInfC{$A^{s+1} \multimap A^{s+1} \vdash \tBool_A$}
\RightLabel{\scriptsize $\multimap L$}
\BinaryInfC{$\tInt_{A^{s+1}} \vdash {}_s \tBool_A$}
\doubleLine\RightLabel{\scriptsize $\multimap L$}
\TrinaryInfC{${}_s \tList_{A^{s+1}} \vdash {}_s \tBool_A$}
\DisplayProof
\end{center}
where the rule $\&L_s$ introduces $s$ new copies of $A$ on the left, such that the original copy is at position $s$ (that is, the last element of the tuple).
\end{proof}

\begin{proposition} \label{decomposing bints taillist}
There exists a proof $\pTail_A$ of ${}_s \tList_{A^{s+1}} \vdash {}_s \tList_A$ which encodes the function $\dntn{\underline{S \sigma}_{A^{s+1}}} \mapsto \dntn{\underline{S}_A}$.
\end{proposition}
\begin{proof}
We could actually give this as a proof of ${}_s \tList_{A^2} \vdash {}_s \tList_A$, but for compatibility with $\pHead_A$ we use the more wasteful type. Define $\pi$ to be the proof
\begin{center}
\AxiomC{}
\UnaryInfC{$A \vdash A$}
\AxiomC{}
\noLine\UnaryInfC{$\cdots$}
\AxiomC{}
\UnaryInfC{$A \vdash A$}
\doubleLine\RightLabel{\scriptsize $\& R$}
\TrinaryInfC{$A \vdash A^{s+1}$}
\AxiomC{}
\UnaryInfC{$A \vdash A$}
\RightLabel{\scriptsize $\& L_0$}
\UnaryInfC{$A^{s+1} \vdash A$}
\RightLabel{\scriptsize $\multimap L$}
\BinaryInfC{$A, A^{s+1} \multimap A^{s+1} \vdash A$}
\RightLabel{\scriptsize $\multimap R$}
\UnaryInfC{$A^{s+1} \multimap A^{s+1} \vdash A \multimap A$}
\DisplayProof
\end{center}
and $\rho$ to be
\begin{center}
\AxiomC{}
\UnaryInfC{$A \vdash A$}
\RightLabel{\scriptsize $\&L_s$}
\UnaryInfC{$A^{s+1} \vdash A$}
\RightLabel{\scriptsize weak}
\UnaryInfC{$A^{s+1}, {!}(A \multimap A) \vdash A$}

\AxiomC{}
\noLine\UnaryInfC{$\cdots$}

\AxiomC{}
\UnaryInfC{$A \vdash A$}
\AxiomC{}
\UnaryInfC{$A \vdash A$}
\RightLabel{\scriptsize $\multimap L$}
\BinaryInfC{$A, A \multimap A \vdash A$}
\RightLabel{\scriptsize $\& L_s$}
\UnaryInfC{$A^{s+1}, A \multimap A \vdash A$}
\RightLabel{\scriptsize der}
\UnaryInfC{$A^{s+1}, {!}(A \multimap A) \vdash A$}

\doubleLine\RightLabel{\scriptsize $\& R$}
\TrinaryInfC{$A^{s+1}, {!}(A \multimap A) \vdash A^{s+1}$}
\RightLabel{\scriptsize $\multimap R$}
\UnaryInfC{${!}(A \multimap A) \vdash A^{s+1} \multimap A^{s+1}$}
\DisplayProof
\end{center}
where the $\cdots$ indicates $s - 1$ branches coming into the $\& R$ rule which are all proofs of $A^{s+1}, {!}(A \multimap A) \vdash A$ identical to the leftmost branch. Finally, $\pTail_A$ is constructed from $s$ copies of $\rho$ and one copy of $\pi$ as in Proposition \ref{decomposing bints tail}.
\end{proof}

Now let $M$ be a Turing machine with tape alphabet $\Sigma = \{0,\ldots,s-1\}$ and define
\[
\tTur^{\Sigma}_A = {!} {}_s \tList_A \otimes {!} {}_s \tList_A \otimes {!} {}_n \tBool_A\,.
\]
We now define a proof
\[
{}_\delta \pStep^{\Sigma}_A: \tTur^{\Sigma}_{A^{s+1}} \vdash \tTur^{\Sigma}_A
\]
encoding the step function of this Turing machine. We have obvious variant of $\pBooltype_A$ which converts ${}_n \tBool_{A^{s+1}}$ to ${}_n \tBool_A$. The first component of the transition function is now encoded as a proof ${}^0_{\delta} \pTrans_A: {}_s \tBool_A, {}_n \tBool_A \vdash {}_s \tBool_A$ which is constructed in the same way, but applying the $\& R$ rule to a family of proofs $\Delta_{0,0}, \ldots, \Delta_{0,s-1}$ encoding respectively $\delta_0(0,-), \ldots, \delta_0(s-1,-)$. Both ${}^1_{\delta} \pTrans_A, {}^2_{\delta} \pTrans_A$ now take as their first input a ${}_s \tBool_A$ rather than a $\tBool_A$, but their constructions are otherwise unchanged. 

Next we need the following straightforward variant of Lemma \ref{appending bools to bints}:

\begin{lemma} \label{appending bools to bintslist}
Let $\cat{W} = \{ W_{ab} \}_{0 \le a,b \le s -1 }$ be an indexed family of sequences $W_{ab} \in \Sigma^*$. There exists a proof $\pi(\cat{W})$ of ${}_s \tList_A, {}_s \tBool_A, {}_s \tBool_A \vdash {}_s \tList_A$ which encodes the function \[(S, \sigma, \tau) \mapsto SW_{\sigma\tau}.\]
\end{lemma}

The proofs ${}^i \pRecomb_A$ of ${}_s \tList_A, 2\,{}_s \tBool_A, \tBool_A \vdash {}_s \tList_A$ are defined in the same way as before using the lemma. From these ingredients we construct proofs
\begin{align*}
{_\delta}\pLeft^{\Sigma}_A:&  \; 3\,{}_s \tList_{A^{s+1}}, \phantom{1}\,{}_s \tList_{A^{s+1}}, 2\,\tNBool_{A^{s+1}} \vdash {_s}\tList_A \\
{_\delta}\pRight^{\Sigma}_A:& \; 2\,{_s}\tList_{A^{s+1}}, 2\,{_s}\tList_{A^{s+1}}, 2\,\tNBool_{A^{s+1}} \vdash {}_s \tList_A \\
{_\delta}\pState^{\Sigma}_A:& \; \phantom{3}\,{_s}\tList_{A^{s+1}},\phantom{1\,{}_s \tList_{A^{s=1}}, 1\,}  \tNBool_{A^{s+1}} \vdash \tNBool_A 
\end{align*}
as before, which give the components of ${}_\delta \pStep^{\Sigma}_A$.

\subsection{Head movement}\label{section:headstill}

In the main text we considered Turing machines in which the head can only move right or left. In this section we add the ability for the head to stay still. The set of directions is then $\{ \textrm{left, right, stay} \}$ which we encode as integers $0, 1, 2$. The transition function is
\[
\delta: \Sigma \times Q \lto \Sigma \times Q \times \{ \textrm{left, right, stay} \}\,.
\]
The type $\tTur_A$ of Turing configurations is unchanged, as are ${}^0_{\delta} \pTrans_A, {}^1_{\delta} \pTrans_A$, but the third transition function is ${^2_\delta}\pTrans_A: \; \tBool_A, \tNBool_A \vdash {}_3 \tBool_A$. The primary change is to the recombination functions:

\begin{proposition}
There exist proofs ${^0}\pRecomb_A$ and ${^1}\pRecomb_A$ of \[\tBint_A, 2\, \tBool_A, {}_3 \tBool_A \vdash \tBint_A\] which encode the functions
\[
(S, \tau, \sigma, d) \mapsto
\begin{cases}
S & \text{if $d = 0$ (left)} \\
S \sigma \tau & \text{if $d = 1$ (right)}\\
S \sigma & \text{if $d = 2$ (stay)}
\end{cases}
\quad \text{and} \quad
(T, \tau, \sigma, d) \mapsto
\begin{cases}
T \tau \sigma & \text{if $d = 0$ (left)} \\
T & \text{if $d = 1$ (right)} \\
T \tau & \text{if $d = 2$ (stay)}
\end{cases}
\]
respectively.
\end{proposition}
\begin{proof}
The desired proof ${^0}\pRecomb_A$ is (with $\Gamma = \tBint_A, 2\, \tBool_A, {!}E, {!}E, A$):
\begin{center}
\AxiomC{$\proofvdots{\pi(\emptyset, \emptyset, \emptyset, \emptyset)}$}
\noLine\UnaryInfC{$\Gamma  \vdash A$}
\AxiomC{$\proofvdots{\pi(00, 10, 01, 11)}$}
\noLine\UnaryInfC{$\Gamma \vdash A$}
\AxiomC{$\proofvdots{\pi(0, 1, 0, 1)}$}
\noLine\UnaryInfC{$\Gamma \vdash A$}
\RightLabel{\scriptsize $\& R$}
\TrinaryInfC{$\tBint_A, 2\, \tBool_A, {!}E, {!}E, A  \vdash A \& A \& A$}
\AxiomC{}
\UnaryInfC{$A \vdash A$}
\RightLabel{\scriptsize $\multimap L$}
\BinaryInfC{$\tBint_A, 2\, \tBool_A, {}_3 \tBool_A, {!}E, {!}E, A \vdash A$}
\doubleLine\RightLabel{\scriptsize $3\times \multimap R$}
\UnaryInfC{$\tBint_A, 2\, \tBool_A, {}_3 \tBool_A \vdash \tBint_A$}
\DisplayProof
\end{center}
and ${^1}\pRecomb_A$ is similar.
\end{proof}

From these ingredients we construct ${}_\delta \pStep_A: \tTur_{A^3} \vdash \tTur_A$ as before.

\section{Ints, Bints and Bools}\label{section:ints_and_bints}

In this appendix we develop the basic properties of denotations of integers, binary integers and booleans in the Sweedler semantics. The denotations of integers and booleans are easily seen to be linearly independent, but binary integers are more interesting.

\begin{proposition} \label{prop-dntnsOfIntsAreLinIndep}
Let $A$ be a formula with $\dim\dntn{A} > 0$. Then
\begin{itemize}
\item[(i)] The set $\{\dntn{\underline{n}_A}\}_{n \geq 0}$ is linearly independent in $\dntn{\tInt_A}$.
\item[(ii)] The set $\{\dntn{\underline{i}_A}\}_{i = 0}^{n-1}$ is linearly independent in $\dntn{{_n} \tBool_A}$.
\end{itemize}
\end{proposition}
\begin{proof}
For (i) suppose that $\sum_{i=0}^n c_i \dntn{\underline{i}} = 0$ for some scalars $c_0, ..., c_n \in \kk$. Let $p \in \kk[x]$ be the polynomial $p = \sum_{i=0}^n c_i x^i$. For $\alpha \in \dntn{A \multimap A}$ we have
\[p(\alpha) = \sum_{i=0}^n c_i \alpha^i = \sum_{i=0}^n c_i \dntn{\underline{i}}\vacu_\alpha  = 0.\]
The minimal polynomial of $\alpha$ therefore divides $p$. Since this holds for any linear map $\alpha \in \dntn{A \multimap A}$, it follows that $p$ is identically zero, as $\kk$ is characteristic zero. The claim in (ii) is clear since the denotations $\dntn{\underline{i}}_A$ are projections from $\dntn{A}^{\oplus n}$.
\end{proof}

\begin{corollary}
The function $\NN \to \dntn{\tInt_A}$ given by $n \mapsto \dntn{\underline{n}_A}$ is injective.
\end{corollary}

We now investigate the question of whether binary integers have linearly independent denotations. Surprisingly this does not hold in general. In fact, for an atomic formula $A$ it is impossible to choose $\dntn{A}$ finite-dimensional such that all binary integers are linearly independent in $\dntn{\tBint_A}$.

For a binary integer $T$, one can write the value of $\dntn{\underline{T}_A}$ on arbitrary kets as a sum of its values on vacuum vectors. This will simplify the task of checking whether binary integers have linearly dependent denotations; at least in the case where we have a fixed number of zeroes and ones, we only need to check linear dependence after evaluating on vacuum vectors. From this point onward let $A$ be a fixed type and write $\dntn{\underline{T}}$ for $\dntn{\underline{T}_A}$.

\begin{proposition}\label{prop-denotationsOfBints}
For $T \in \{0,1\}^n$, let $t_i$ be the $i$th digit of $T$ and let $T_0 = \{ i \l t_i = 0\}$ and $T_1 = \{ i \l t_i = 1\}$. Then
\[
\dntn{\underline{T}_A}(\ket{\alpha_1, ..., \alpha_k}_{\gamma}\otimes \ket{\beta_1, ..., \beta_l}_{\delta}) = \Sum_{f \in \Inj([k], T_0)}\Sum_{g \in \Inj([l], T_1)} \Gamma_{n}^{f,g} \circ \cdots \circ \Gamma_{1}^{f,g}.
\]
where 
\[ 
\Gamma^{f,g}_i = \begin{cases}
\alpha_{f^{-1}(i)} & i \in \im(f) \\
\beta_{g^{-1}(i)} & i \in \im(g) \\
\gamma & i \in T_0 \setminus \im(f) \\
\delta & i \in T_1 \setminus \im(g). 
\end{cases}
\] 
In particular, $\dntn{\underline{T}_A}$ vanishes on $\ket{\alpha_1, ..., \alpha_k}_{\gamma}\otimes \ket{\beta_1, ..., \beta_l}_{\delta}$ if $k > |T_0|$ or $l > |T_1|$.
\end{proposition}
\begin{proof}
This follows in the same way as \cite[Lemma 3.11]{clift_murfet}.
\end{proof}

\begin{proposition}
Let $m,n \geq 0$, and let $T \in \{0,1\}^*$ contain exactly $m$ zeroes and $n$ ones. Then for $\alpha_i, \beta_i, \gamma, \delta \in \dntn{A \multimap A}$ we have:
{\small\[ 
\dntn{\underline{T}}(\ket{\alpha_1, ..., \alpha_m}_{\gamma}, \ket{\beta_1, ..., \beta_n}_{\delta}) = \sum_{I \subseteq [m]}\Sum_{J \subseteq [n]}  (-1)^{m-|I|}(-1)^{n-|J|} \dntn{\underline{T}}\left(\vacu_{\Sum_{i \in I} \alpha_i}, \vacu_{\Sum_{j \in J} \beta_j}\right).\tag{$*$}
\]}
Note that the right hand side does not depend on $\gamma$ and $\delta$.
\end{proposition}

\begin{proof}
A single term of the double sum on the right hand side of $(*)$ corresponds to replacing (in the reversal of $T$) each $0$ with $\sum_i \alpha_i$ and each $1$ with $\sum_j \beta_j$. After expanding these sums, consider the coefficient of a particular noncommutative monomial $p$ which contains only the variables $\alpha_{i_1}, ..., \alpha_{i_k}, \beta_{j_1}, ..., \beta_{j_l}$, where $\{i_1, ..., i_k\} \subseteq [m]$ and $\{j_1, ..., j_l\} \subseteq [n]$. The terms in the right hand side of $(*)$ which contribute to this coefficient are precisely those for which the set $I$ contains each of $i_1, ..., i_k$, and $J$ contains each of $j_1, ..., j_l$. Hence the coefficient of $p$ is
\begin{align*}
\sum_{\substack{\{i_1, ..., i_k\} \subseteq I \subseteq [m] \\ \{j_1, ..., j_l\} \subseteq J \subseteq [n]}} (-1)^{m-|I|}(-1)^{n-|J|}
&= \sum_{i=k}^{m}\sum_{j=l}^{n} (-1)^{m-i}(-1)^{n-j} \binom{m-k}{i-k}\binom{n-l}{j-l} 
\\&= \sum_{i=k}^{m} (-1)^{m-i}\binom{m-k}{i-k}\sum_{j=l}^{n}(-1)^{n-j} \binom{n-l}{j-l} 
\\&= \begin{cases}1 & m = k \text{ and } n = l \\ 0 & \text{otherwise}.\end{cases}
\end{align*}
So the only monomials $p$ with non-zero coefficient are those which contain each $\alpha_i$ and each $\beta_j$ exactly once, which agrees with the left hand side of $(*)$ by Proposition \ref{prop-denotationsOfBints}.
\end{proof}

In order to make the content of the above proposition clear, we give a concrete example involving a specific binary sequence.
\begin{example}
Let $T = 0010$. The left hand side of $(*)$ is therefore
\[
\sum_{\sigma \in S_3} \alpha_{\sigma(1)}\beta_1\alpha_{\sigma(2)}\alpha_{\sigma(3)}
\]
while the right hand side is
\begin{align*}
&(\alpha_1 + \alpha_2 + \alpha_3)\beta_1(\alpha_1 + \alpha_2 + \alpha_3)(\alpha_1 + \alpha_2 + \alpha_3) - (\alpha_1 + \alpha_2)\beta_1(\alpha_1 + \alpha_2)(\alpha_1 + \alpha_2) - \\ &(\alpha_1 + \alpha_3)\beta_1(\alpha_1 + \alpha_3)(\alpha_1 + \alpha_3) - (\alpha_2 + \alpha_3)\beta_1(\alpha_2 + \alpha_3)(\alpha_2 + \alpha_3) + \\&\alpha_1 \beta_1 \alpha_1 \alpha_1 + \alpha_2 \beta_1 \alpha_2 \alpha_2 + \alpha_3 \beta_1 \alpha_3 \alpha_3.
\end{align*}
After expanding the above, consider for example the monomial $\alpha_{1} \beta_1 \alpha_{1} \alpha_{1}$:
\begin{align*}
&(\color{red}\alpha_1\color{black} + \alpha_2 + \alpha_3)\color{red}\beta_1\color{black}(\color{red}\alpha_1\color{black} + \alpha_2 + \alpha_3)(\color{red}\alpha_1\color{black} + \alpha_2 + \alpha_3) - (\color{red}\alpha_1\color{black} + \alpha_2)\color{red}\beta_1\color{black}(\color{red}\alpha_1\color{black} + \alpha_2)(\color{red}\alpha_1\color{black} + \alpha_2) - \\ &(\color{red}\alpha_1\color{black} + \alpha_3)\color{red}\beta_1\color{black}(\color{red}\alpha_1\color{black} + \alpha_3)(\color{red}\alpha_1\color{black} + \alpha_3) - (\alpha_2 + \alpha_3)\beta_1(\alpha_2 + \alpha_3)(\alpha_2 + \alpha_3) + \\&\color{red}\alpha_1\color{black} \color{red}\beta_1\color{black}\color{red}\alpha_1\color{black} \color{red}\alpha_1\color{black} + \alpha_2 \beta_1 \alpha_2 \alpha_2 + \alpha_3 \beta_1 \alpha_3 \alpha_3.
\end{align*}
As expected, the coefficient is $\binom{2}{2} - \binom{2}{1} + \binom{2}{0} = 1 - 2 + 1 = 0$.
\end{example}

\begin{corollary}\label{corol-evaluation-of-bints-on-kets-can-be-written-as-sum-on-vacuums}
Let $m,n \geq 0$, let $T \in \{0,1\}^*$ contain exactly $m$ zeroes and $n$ ones, and let $0 \leq k \leq m$ and $0 \leq l \leq n$. Then
\[ 
\dntn{\underline{T}}(\ket{\alpha_1, ..., \alpha_k}_{\gamma}, \ket{\beta_1, ..., \beta_l}_{\delta}) = \Sum_{I \subseteq [m]} \Sum_{J \subseteq [n]}  \frac{(-1)^{m-|I|}(-1)^{n-|J|}}{(m-k)!(n-l)!}\dntn{\underline{T}}\left(\vacu_{\Sum_{i \in I} \alpha_i}, \vacu_{\Sum_{j \in J} \beta_j}\right),
\] where on the right hand side we define $\alpha_i = \gamma$ for $i > k$ and $\beta_j = \delta$ for $j > l$.
\end{corollary}

\begin{proof}
Note that by Proposition \ref{prop-denotationsOfBints}\[\dntn{\underline{T}}(\ket{\alpha_1, ..., \alpha_k}_{\gamma}, \ket{\beta_1, ..., \beta_l}_{\delta}) = \frac{\dntn{\underline{T}}(\ket{\alpha_1, ..., \alpha_k, \gamma, ..., \gamma}_{\gamma}, \ket{\beta_1, ..., \beta_l, \delta, ..., \delta}_{\delta})}{(m-k)!(n-l)!},\] where the kets on the right have length $m$ and $n$ respectively. Then apply the previous proposition.
\end{proof}

\begin{lemma} \label{lemma-bints-lincomb-equalzero-is-equiv-to-bints-lincomb-on-vacuums-equalzero}
Let $T_1, ..., T_r \in \{0,1\}^*$ each contain exactly $m$ zeroes and $n$ ones, and let $c_1, ..., c_r \in \kk \setminus\{0\}$. Then $\sum_s c_s \dntn{\underline{T_s}} = 0$ in $\dntn{\tBint_A}$ if and only if $\sum_s c_s\dntn{\underline{T_s}}(\vacu_{\alpha}, \vacu_{\beta}) = 0$ in $\dntn{A \multimap A}$ for all $\alpha, \beta \in \dntn{A \multimap A}$.
\end{lemma}

\begin{proof}
$(\Rightarrow)$ is immediate. For $(\Leftarrow)$, suppose that $\sum_{s=1}^r c_s\dntn{\underline{T_s}}(\vacu_{\alpha}, \vacu_{\beta}) = 0$ for all $\alpha, \beta$. It follows from Corollary \ref{corol-evaluation-of-bints-on-kets-can-be-written-as-sum-on-vacuums} that
\begin{align*}
&\phantom{=\,} \sum_{s=1}^r c_s \dntn{\underline{T_s}}(\ket{\alpha_1, ..., \alpha_k}_{\gamma}, \ket{\beta_1, ..., \beta_l}_{\delta})
\\&= \Sum_{I \subseteq [m]}\Sum_{J \subseteq [n]}\frac{(-1)^{m-|I|} (-1)^{n-|J|}}{(m-k)!(n-l)!} \sum_{s=1}^r c_s \dntn{\underline{T_s}}\left(\vacu_{\Sum_{i \in I} \alpha_i}, \vacu_{\Sum_{j \in J} \beta_j}\right)
\\&= \Sum_{I \subseteq [m]}\Sum_{J \subseteq [n]}\frac{(-1)^{m-|I|} (-1)^{n-|J|}}{(m-k)!(n-l)!} \cdot 0 \\&=0,
\end{align*}
where we again define $\alpha_i = \gamma$ for $i>k$ and $\beta_j = \delta$ for $j>l$.
\end{proof}

Suppose that $A$ is a formula with $\dim\dntn{A} = n < \infty$. By the above lemma, the existence of distinct binary integers with linearly dependent denotations reduces to the task of finding a non-zero noncommutative polynomial $t_n(x,y)$ such that $t_n(\alpha, \beta) = 0$ for all $n \times n$ matrices $\alpha, \beta \in \dntn{A \multimap A}$. To describe such a polynomial, we will require the following theorem.

\begin{theorem}
\textbf{(Amitsur-Levitzki Theorem.)} For $n \in \NN$, let $\kk\<x_1, ..., x_n\>$ denote the ring of noncommutative polynomials in $n$ variables, and let $s_n \in \kk\<x_1, ..., x_n\>$ be the polynomial
\[
s_n = \sum_{\sigma \in S_{n}} \sgn(\sigma) x_{\sigma(1)} \cdots x_{\sigma(n)}.
\]
Then for all $\alpha_1, ..., \alpha_{2n} \in M_n(\kk)$, we have $s_{2n}(\alpha_1, ..., \alpha_{2n}) = 0$. Furthermore, $M_n(\kk)$ does not satisfy any polynomial identity of degree less than $2n$. 
\end{theorem}
\begin{proof}
See \cite[Theorem 3.1.4]{drensky-formanek-polynomialIdentityRings}.
\end{proof}

\begin{corollary}
For $n \in \NN$, there exists a non-zero polynomial $t_n \in \kk\<x,y\>$ such that for all $\alpha, \beta \in M_n(\kk)$ we have $t_n(\alpha,\beta) = 0$.
\end{corollary}
\begin{proof}
The polynomial $t_n(x,y) = s_{2n}(x, xy, ..., xy^{2n-1})$ is non-zero and satisfies the desired property. 
\end{proof}

\begin{proposition} \label{prop: finite dim bints are lindep}
For any formula $A$ such that $\dim\dntn{A} < \infty$, there exist distinct binary integers $T_1, ..., T_r \in \{0,1\}^*$ such that $\dntn{\underline{T_1}}, ..., \dntn{\underline{T_r}}$ are linearly dependent in $\dntn{\tBint_A}$.
\end{proposition}

\begin{proof}
Let $n = \dim\dntn{A}$, so that $\dntn{A \multimap A} \cong M_n(k)$. For $1 \leq i \leq 2n$, let $R_i$ be the binary integer $1^{i-1}0$. Note that for all $\alpha, \beta \in M_n(k)$ we have
\begin{align*}
\sum_{\sigma \in S_{2n}} \sgn(\sigma)\dntn{\underline{R_{\sigma(2n)} \cdots R_{\sigma(1)}}}(\vacu_\alpha, \vacu_\beta)
= t_n(\alpha, \beta)
= 0.
\end{align*}
Hence $\sum_{\sigma \in S_{2n}} \sgn(\sigma)\dntn{\underline{R_{\sigma(2n)} \cdots R_{\sigma(1)}}} = 0$ by Lemma \ref{lemma-bints-lincomb-equalzero-is-equiv-to-bints-lincomb-on-vacuums-equalzero}.
\end{proof}

\begin{remark}\label{remark:findgoodA}
Note that despite the above proposition, if we have a \textit{particular} finite collection of binary integers $T_1, ..., T_r$ in mind it is always possible for $A$ atomic to choose $\dntn{A}$ such that $\dntn{\underline{T_1}}, ..., \dntn{\underline{T_r}}$ are linearly independent in $\dntn{\tBint_A}$. To see this, let $d$ denote the maximum length of the $T_s$, and note that a linear dependence relation between the $\dntn{\underline{T_s}}$ gives rise to a polynomial identity for $M_n(\kk)$ of degree $d$, where $n = \dim\dntn{A}$. By Amitsur-Levitzki we must therefore have $d \geq 2n$, so if we choose $\dim\dntn{A} > d/2$ then $\dntn{\underline{T_1}}, ..., \dntn{\underline{T_r}}$ must be linearly independent.
\end{remark}

In addition, while linear independence does not always hold for an arbitrary collection of binary integers, it turns out that we do have linear independence for any \textit{pair} of distinct binary integers, so long as $\dim\dntn{A}$ is at least 2. 

\begin{proposition}
Let $A$ be a formula with $\dim\dntn{A} \geq 2$. The function $\{0,1\}^* \to \dntn{\tBint_A}$ which maps $S$ to $\dntn{\underline{S}}$ is injective.
\end{proposition}
\begin{proof}
Let $n = \dim\dntn{A}$. For simplicity of notation we suppose that $n$ is finite, as the case where $n$ is infinite is similar. Consider the subgroup $G$ of $GL_n(\kk)$ generated by
\[
\alpha = \begin{bmatrix}
1 & 2 & 0 & \cdots & 0 \\ 0 & 1 & 0 & \cdots & 0 \\ 0 & 0 & 1 & \cdots & 0 \\ \vdots & \vdots & \vdots & \ddots & \vdots \\ 0 & 0 & 0 & \cdots & 1
\end{bmatrix}
\text{ and }
\beta = \begin{bmatrix}
1 & 0 & 0 & \cdots & 0 \\ 2 & 1 & 0 & \cdots & 0 \\ 0 & 0 & 1 & \cdots & 0 \\ \vdots & \vdots & \vdots & \ddots & \vdots \\ 0 & 0 & 0 & \cdots & 1
\end{bmatrix}.
\]
It is well known that $G$ is freely generated by $\alpha$ and $\beta$; see \cite[\S II.B]{harpe-topicsFromGeometricGroupTheory}. Suppose that $\dntn{\underline{S}} = \dntn{\underline{T}}$, so that in particular we have $\dntn{\underline{S}}(\vacu_\alpha, \vacu_\beta) = \dntn{T}(\vacu_\alpha, \vacu_\beta)$.  In other words, the composite obtained by substituting $\alpha$ for zero and $\beta$ for one into the digits of $S$ is equal to the corresponding composite for $T$. Since $\alpha$ and $\beta$ generate a free group, it follows that $S = T$.
\end{proof}

\begin{proposition}
Let $A$ be a formula with $\dim\dntn{A} \geq 2$, and let $S, T \in \{0,1\}^*$ with $S \neq T$. The denotations $\dntn{\underline{S}}, \dntn{\underline{T}}$ are linearly independent in $\dntn{\tBint_A}$.
\end{proposition}

\begin{proof}
Suppose we are given $S,T \in \{0,1\}^*$ such that $a \dntn{\underline{S}} + b \dntn{\underline{T}} = 0$ for some $a,b \neq 0$. With $\alpha, \beta$ as above, it follows that
\[
\dntn{\underline{S}}(\vacu_\alpha, \vacu_\beta) \circ \dntn{\underline{T}}(\vacu_\alpha, \vacu_\beta)^{-1} = -\frac{b}{a}I
\]
So $\dntn{\underline{S}}(\vacu_\alpha, \vacu_\beta) \circ \dntn{\underline{T}}(\vacu_\alpha, \vacu_\beta)^{-1}$ is in the center of $G$, which is trivial since $G$ is free of rank $2$, and hence $a=-b$. It follows that $\dntn{\underline{S}} = \dntn{\underline{T}}$ and therefore $S = T$ by the previous proposition.
\end{proof}

\bibliographystyle{amsalpha}
\providecommand{\bysame}{\leavevmode\hbox to3em{\hrulefill}\thinspace}
\providecommand{\href}[2]{#2}

\end{document}